\documentclass[11pt]{article}
\usepackage{etex}

\usepackage[margin=0.9in]{geometry}
\usepackage{stackengine}

\usepackage{stmaryrd}

\usepackage{tabularx}

\usepackage{color}
\usepackage[latin1]{inputenc}
\usepackage[T1]{fontenc}
\usepackage[normalem]{ulem}
\usepackage[english]{babel}
\usepackage{verbatim}
\usepackage{graphicx}
\usepackage{enumerate}
\usepackage{amsmath,amssymb,amsfonts,amsthm,amscd,mathrsfs}
\usepackage{array}

\usepackage{dsfont}

\usepackage[T1]{fontenc}
\usepackage{babel}

\usepackage{url}

\usepackage{caption}
\usepackage{subcaption}
\usepackage[bookmarksopen, bookmarksnumbered]{hyperref}

\usepackage{tikz}
\usetikzlibrary{arrows}
\usetikzlibrary{arrows.meta}
\definecolor{wwhhii}{rgb}{1.,1.,1.}
\definecolor{rreedd}{rgb}{1.,0.,0.}
\definecolor{uuuuuu}{rgb}{0.26666666666666666,0.26666666666666666,0.26666666666666666}
\definecolor{darkgreen}{HTML}{0d8513}

\newtheorem{theorem}{Theorem}[section]

\newtheorem{lemma}[theorem]{Lemma}

\newtheorem{prop}[theorem]{Proposition}

\theoremstyle{definition}

\newtheorem{rem}[theorem]{Remark}

\input xy
\xyoption{all}

\DeclareMathOperator{\Exp}{Exp}

\DeclareMathOperator{\TV}{TV}

\DeclareMathOperator{\GOE}{GOE}

\DeclareMathOperator{\Ai}{Ai}

\DeclareMathOperator{\pr}{prob}

\newcommand{\PP}{\mathbb P}
\newcommand{\Z}{\mathbb Z}
\newcommand{\R}{\mathbb R}

\newcommand{\N}{\mathbb N}

\newcommand{\don}{\mathds{1}}

\newcommand{\cD}{\mathcal D}
\newcommand{\cP}{\mathcal P}

\newcommand{\cB}{\mathcal B}

\newcommand{\cE}{\mathcal E}
\newcommand{\cA}{\mathcal A}

\newcommand{\cM}{\mathcal M}
\newcommand{\cEo}{\dot{\cE}}
\newcommand{\cEd}{\ddot{\cE}}
\newcommand{\cW}{\mathcal W}
\newcommand{\ocE}{\overline{\cE}}
\newcommand{\cH}{\mathcal H}

\newcommand{\cQ}{\mathcal Q}

\newcommand{\fh}{\mathfrak h}

\newcommand{\ii}{\mathfrak i}

\newcommand{\etao}{\dot{\zeta}}
\newcommand{\etad}{\ddot{\zeta}}
\newcommand{\bxi}{\boldsymbol{\xi}}
\newcommand{\bzeta}{\boldsymbol{\zeta}}
\newcommand{\btad}{\boldsymbol{\etad}}
\newcommand{\btao}{\boldsymbol{\etao}}
\newcommand{\bta}{\boldsymbol{\eta}}
\newcommand{\bla}{\boldsymbol{\lambda}}

\newcommand{\uetao}{\underline{\etao}}
\newcommand{\uetad}{\underline{\etad}}

\newcommand{\oxi}{\overline{\xi}}
\newcommand{\obxi}{\overline{\bxi}}

\newcommand{\hT}{\hat{T}}

\newcommand\doeq{\mathrel{\stackrel{\makebox[0pt]{\mbox{\normalfont\tiny d}}}{=}}}

\makeatletter
\renewcommand\tableofcontents{%
  \null\hfill\textbf{\Large\contentsname}\hfill\null\par
  \@mkboth{\MakeUppercase\contentsname}{\MakeUppercase\contentsname}%
  \@starttoc{toc}%
}

\g@addto@macro\normalsize{%
  \setlength\abovedisplayskip{5pt}
  \setlength\belowdisplayskip{5pt}
  \setlength\abovedisplayshortskip{3pt}
  \setlength\belowdisplayshortskip{3pt}
}

\makeatother

\numberwithin{equation}{section}

\begin{document}
\title{Cutoff profile of the Metropolis biased card shuffling}

\author{Lingfu Zhang
\thanks{Department of Statistics, University of California, Berkeley. e-mail: lfzhang@berkeley.edu}
}
\date{}

\maketitle

\begin{abstract}
We consider the Metropolis biased card shuffling (also called the multi-species ASEP on a finite interval or the random Metropolis scan).
Its convergence to stationary was believed to exhibit a total-variation cutoff, and that was proved a few years ago by Labb{\'e} and Lacoin \cite{labbe2019cutoff}.
In this paper, we prove that (for $N$ cards) the cutoff window is in the order of $N^{1/3}$, and the cutoff profile is given by the GOE Tracy-Widom distribution function.
This confirms a conjecture by Bufetov and Nejjar \cite{bufetov2022cutoff}. Our approach is different from \cite{labbe2019cutoff}, by comparing the card shuffling with the multi-species ASEP on $\Z$, and using Hecke algebra and recent ASEP shift-invariance and convergence results.
Our result can also be viewed as a generalization of the Oriented Swap Process finishing time convergence \cite{bufetov2022absorbing}, which is the TASEP version (of our result).
\end{abstract}

\section{Introduction} \label{sec:intro}
In this paper we study the \emph{(Metropolis) biased card shuffling} of size $N$, which can be viewed as a continuous time Markov chain with state space being the permutation group $S_N$.
This Markov chain can be described as follows.
Let $[N]$ denote the set $\{1,\ldots,N\}$, and let elements in $S_N$ be represented by bijections from $[N]$ to itself.
Suppose that the current state is $\lambda:[N]\to[N]$, then for each $i\in[N-1]$ independently, if $\lambda(i)<\lambda(i+1)$, with rate $1$ we swap $\lambda(i)$ and $\lambda(i+1)$; otherwise, with rate $q$ we swap $\lambda(i)$ and $\lambda(i+1)$.
Here $q\in [0,1)$ is a constant.
This Markov chain is known to have a stationary measure, $\cM_N$, which is also called the \emph{Mallows measure} of size $N$.
We let $\cW_{N,t}^\lambda$ be the law of this chain at time $t$, starting from a deterministic state $\lambda$.
Then $\cW_{N,t}^\lambda\to \cM_N$ as $t\to\infty$.
Our main result concerns the convergence of $\cW_{N,t}^\lambda$ to $\cM_N$ in the total-variation distance, which is defined as
\[
\|\cP-\cQ\|_{\TV}=\max_{A\subset S_N}|\cP(A)-\cQ(A)|
\]
for any two probability measures $\cP$ and $\cQ$ of $S_N$.
\begin{theorem}  \label{thm:main}
For any $\tau \in \R$, we have
\[
\lim_{N\to\infty}\max_{\lambda\in S_N} \|\cW_{N,2(1-q)^{-1}(N+\tau N^{1/3})}^\lambda-\cM_N\|_{\TV} = 1-F_{\GOE}(2^{2/3}\tau),
\]
where $F_{\GOE}$ is the distribution function of the Tracy-Widom GOE distribution.
\end{theorem}
The Tracy-Widom GOE distribution was introduced in \cite{tracy1996orthogonal}, where it was shown to govern the fluctuation of the largest eigenvalue in the Gaussian Orthogonal Ensemble (GOE) random matrices.
It can be defined as
\begin{equation} \label{eq:goedef}
F_{\GOE}(s)=e^{-2^{-1}\int_s^\infty q(x)+(x-s)^2q(x)^2dx},
\end{equation}
where $p(x)$ is the Hastings-McLeod solution of the second Plainlev\'{e} equation $p''(x)=2p(x)^3+xp(x)$, with boundary condition $p(x)\sim\Ai(x)$ as $x\to\infty$ for $\Ai$ being the Airy function.

Our main result (Theorem \ref{thm:main}) basically says that there is a cutoff phenomenon in the total-variation convergence of the biased card shuffling, with cutoff window in the order of $N^{1/3}$, and the cutoff profile is GOE Tracy-Widom.

\subsection{Background and related works}

The word `cutoff' was first used by Aldous and Diaconis \cite{aldous1986shuffling} to describe the phenomenon, where (usually in terms of the total-variation distance) a Markov chain stays away from its stationary distribution for some time, but then converges abruptly;
see the review articles \cite{diaconis1996cutoff, saloff2004random} and the textbook \cite{wilmer2009markov}.
To rigorously establish cutoff often requires rather careful understandings of the particular Markov chains, as achieved in many instances in the past decades, see e.g. \cite{diaconis1981generating, aldous1983random, bayer1992trailing, lubetzky2010cutoff, ding2010total, lubetzky2013cutoff, lubetzky2016cutoff, labbe2019cutoff, salez2021cutoff} where cutoff was proved for various (families of) Markov chains using a variety of methods, and many of them were on different card shuffling processes.

The `cutoff window' refers to the small period when the total-variation decays from near $1$ to near $0$, and the `cutoff profile' describes the shape of the decay within the window.
These are refined information of Markov chains, and are not available even for most Markov chains where cutoff have been established. 
Among those cases where cutoff profiles are known, many of them are given by the distribution function of Gaussian random variables, see e.g. \cite{lacoin2016cutoff, lubetzky2016cutoff, ben2017cutoff} (and also \cite{chatterjee2022phase} where cutoff is in a slightly different setting). 
Some other types of cutoff profiles have also been discovered; see e.g. \cite{teyssier2020limit,nestoridi2022limit,bufetov2022cutoff}.

The biased card shuffling, the Markov chain we analyze in this paper, was first studied by Diaconis and Ram in \cite{diaconis2000analysis} under the name `random Metropolis scan', as it is an instance of the Metropolis algorithm that samples biased random permutations.
The mixing of this Markov chain was later studied in \cite{benjamini2005mixing}, whose main result states that the mixing time is at most linear in the size $N$. 
A simpler proof of this result later appeared in \cite{greenberg2009sampling}.
Since then, whether cutoff happens for the biased card shuffling remained an open problem.
This was answered affirmatively by Labb{\'e} and Lacoin \cite{labbe2019cutoff}, who proved that the total-variation cutoff happens at time $2(1-q)^{-1}N$. In the same paper the spectral gap was also computed.\\

\noindent\textbf{Simple Exclusion Process.} A closely related Markov chain is the \emph{Asymmetric Simple Exclusion Process} (ASEP), which has been used as a key tool in studying the biased card shuffling.
It is a classical interacting particle system, and has been intensively studied in the literature (back to e.g. \cite{liggett1985interacting}).
It can be formulated as follows.
For a (finite or infinite) sequence of sites (indexed by $\Z$ or a finite discrete interval), each site may be occupied by a particle or be empty (later we shall represent this by a function from $\Z$ or the finite discrete interval to $\{0, 1\}$, with $1$ denoting a particle, and $0$ denoting a hole).
For each particle, if the site next to it in the right is empty, with rate $1$ it jumps to the right; and if the site next to it in the left is empty, with rate $q$ it jumps to the left.
Such jumps happen independently for each particle in each direction.
The ASEP is an important model in the so-called  Kardar-Parisi-Zhang (KPZ) universality class, and has exact-solvable structures.
The hydrodynamics are well understood, and its diffusive scaling limit has been identified with the KPZ equation \cite{bertini1997stochastic}.
However, the scaling invariant and long time limit of the ASEP should be obtained under the KPZ scaling instead. That limit is termed the \emph{KPZ fixed point}, and was first constructed in \cite{matetski2016kpz} as the KPZ scaling limit in the \emph{totally asymmetric} case (i.e. when $q=0$), and later proved for the general $q\in [0,1)$ case in \cite{quastel2022convergence}.

\begin{figure}[hbt!]
    \centering
\begin{tikzpicture}[line cap=round,line join=round,>=triangle 45,x=7cm,y=7cm]
\clip(-3,-0.2) rectangle (-0.5,.63);

\draw [line width=1.2pt, opacity=0.3] (-2.5,0.3) -- (-0.7,0.3);
\draw [fill=white] (-2.4,0.3) circle (6.0pt);
\draw [fill=white] (-2.3,0.3) circle (6.0pt);
\draw [fill=uuuuuu] (-2.2,0.3) circle (6.0pt);
\draw [fill=uuuuuu] (-2.1,0.3) circle (6.0pt);
\draw [fill=white] (-2.,0.3) circle (6.0pt);
\draw [fill=uuuuuu] (-1.9,0.3) circle (6.0pt);
\draw [fill=white] (-1.8,0.3) circle (6.0pt);
\draw [fill=white] (-1.7,0.3) circle (6.0pt);
\draw [fill=white] (-1.6,0.3) circle (6.0pt);
\draw [fill=white] (-1.5,0.3) circle (6.0pt);
\draw [fill=white] (-1.4,0.3) circle (6.0pt);
\draw [fill=uuuuuu] (-1.3,0.3) circle (6.0pt);
\draw [fill=uuuuuu] (-1.2,0.3) circle (6.0pt);
\draw [fill=white] (-1.1,0.3) circle (6.0pt);
\draw [fill=uuuuuu] (-1.,0.3) circle (6.0pt);
\draw [fill=white] (-0.9,0.3) circle (6.0pt);
\draw [fill=uuuuuu] (-0.8,0.3) circle (6.0pt);

\draw [line width=1.2pt, opacity=0.3] (-2.5,0.1) -- (-0.7,0.1);
\draw [fill=uuuuuu] (-2.4,0.1) circle (6.0pt);
\draw [fill=white] (-2.3,0.1) circle (6.0pt);
\draw [fill=uuuuuu] (-2.2,0.1) circle (6.0pt);
\draw [fill=uuuuuu] (-2.1,0.1) circle (6.0pt);
\draw [fill=white] (-2.,0.1) circle (6.0pt);
\draw [fill=uuuuuu] (-1.9,0.1) circle (6.0pt);
\draw [fill=white] (-1.8,0.1) circle (6.0pt);
\draw [fill=white] (-1.7,0.1) circle (6.0pt);
\draw [fill=white] (-1.6,0.1) circle (6.0pt);
\draw [fill=white] (-1.5,0.1) circle (6.0pt);
\draw [fill=white] (-1.4,0.1) circle (6.0pt);
\draw [fill=uuuuuu] (-1.3,0.1) circle (6.0pt);
\draw [fill=uuuuuu] (-1.2,0.1) circle (6.0pt);
\draw [fill=white] (-1.1,0.1) circle (6.0pt);
\draw [fill=uuuuuu] (-1.,0.1) circle (6.0pt);
\draw [fill=white] (-0.9,0.1) circle (6.0pt);
\draw [fill=uuuuuu] (-0.8,0.1) circle (6.0pt);

\draw [line width=1.2pt, opacity=0.3] (-2.5,-0.1) -- (-0.7,-0.1);
\draw [fill=uuuuuu] (-2.4,-0.1) circle (6.0pt);
\draw [fill=white] (-2.3,-0.1) circle (6.0pt);
\draw [fill=uuuuuu] (-2.2,-0.1) circle (6.0pt);
\draw [fill=uuuuuu] (-2.1,-0.1) circle (6.0pt);
\draw [fill=white] (-2.,-0.1) circle (6.0pt);
\draw [fill=uuuuuu] (-1.9,-0.1) circle (6.0pt);
\draw [fill=uuuuuu] (-1.8,-0.1) circle (6.0pt);
\draw [fill=white] (-1.7,-0.1) circle (6.0pt);
\draw [fill=white] (-1.6,-0.1) circle (6.0pt);
\draw [fill=white] (-1.5,-0.1) circle (6.0pt);
\draw [fill=white] (-1.4,-0.1) circle (6.0pt);
\draw [fill=uuuuuu] (-1.3,-0.1) circle (6.0pt);
\draw [fill=uuuuuu] (-1.2,-0.1) circle (6.0pt);
\draw [fill=white] (-1.1,-0.1) circle (6.0pt);
\draw [fill=uuuuuu] (-1.,-0.1) circle (6.0pt);
\draw [fill=white] (-0.9,-0.1) circle (6.0pt);
\draw [fill=uuuuuu] (-0.8,-0.1) circle (6.0pt);

\draw [line width=1.2pt, opacity=0.3] (-2.5,0.5) -- (-0.7,0.5);
\draw [fill=white] (-2.4,0.5) circle (6.0pt);
\draw [fill=white] (-2.3,0.5) circle (6.0pt);
\draw [fill=white] (-2.2,0.5) circle (6.0pt);
\draw [fill=white] (-2.1,0.5) circle (6.0pt);
\draw [fill=white] (-2.,0.5) circle (6.0pt);
\draw [fill=white] (-1.9,0.5) circle (6.0pt);
\draw [fill=white] (-1.8,0.5) circle (6.0pt);
\draw [fill=white] (-1.7,0.5) circle (6.0pt);
\draw [fill=white] (-1.6,0.5) circle (6.0pt);
\draw [fill=white] (-1.5,0.5) circle (6.0pt);
\draw [fill=white] (-1.4,0.5) circle (6.0pt);
\draw [fill=white] (-1.3,0.5) circle (6.0pt);
\draw [fill=white] (-1.2,0.5) circle (6.0pt);
\draw [fill=white] (-1.1,0.5) circle (6.0pt);
\draw [fill=white] (-1.,0.5) circle (6.0pt);
\draw [fill=white] (-0.9,0.5) circle (6.0pt);
\draw [fill=white] (-0.8,0.5) circle (6.0pt);

\draw (-2.5,0.3) node[anchor=east]{particles $\le 7$};
\draw (-2.5,0.1) node[anchor=east]{particles $\le 8$};
\draw (-2.5,-0.1) node[anchor=east]{particles $\le 9$};

\begin{scriptsize}
\draw (-2.4,0.47) node[anchor=south]{$8$};
\draw (-2.3,0.47) node[anchor=south]{$16$};
\draw (-2.2,0.47) node[anchor=south]{$4$};
\draw (-2.1,0.47) node[anchor=south]{$3$};
\draw (-2.,0.47) node[anchor=south]{$17$};
\draw (-1.9,0.47) node[anchor=south]{$5$};
\draw (-1.8,0.47) node[anchor=south]{$9$};
\draw (-1.7,0.47) node[anchor=south]{$10$};
\draw (-1.6,0.47) node[anchor=south]{$11$};
\draw (-1.5,0.47) node[anchor=south]{$14$};
\draw (-1.4,0.47) node[anchor=south]{$12$};
\draw (-1.3,0.47) node[anchor=south]{$2$};
\draw (-1.2,0.47) node[anchor=south]{$1$};
\draw (-1.1,0.47) node[anchor=south]{$15$};
\draw (-1.,0.47) node[anchor=south]{$7$};
\draw (-.9,0.47) node[anchor=south]{$13$};
\draw (-.8,0.47) node[anchor=south]{$6$};
\end{scriptsize}

\end{tikzpicture}
\caption{
An illustration of projecting the biased card shuffling into single-species ASEPs.
}  
\label{fig:biap}
\end{figure}
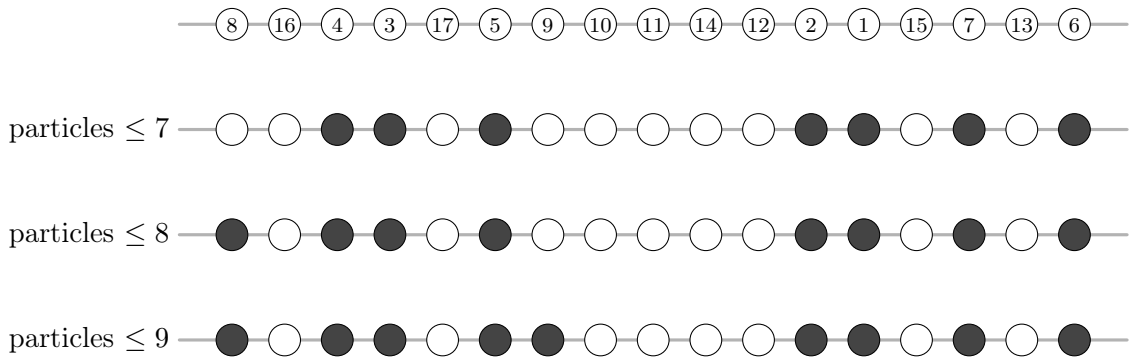

The ASEP on the finite discrete interval $[N]$ can be easily connected to the biased card shuffling of size $N$.
Namely, we can get the ASEP on $[N]$ with $k$ particles, if in the biased card shuffling we replace each number $\le k$ by a particle and each number $>k$ by a hole (see Figure \ref{fig:biap}).
In other words, we can project the biased card shuffling of size $N$ into the ASEP, for each integer $k$ between $0$ and $N$. It is worth mentioning that given all such projections (for all $k$), we can reconstruct the biased card shuffling (see Section \ref{ssec:mset} below for the details).
For this reason we also refer to the biased card shuffling of size $N$ as `the multi-species ASEP' or `the colored ASEP' on $[N]$.
As the ordinary single-species ASEP, the multi-species ASEP can also be defined on $\Z$ or any finite discrete interval (not just $[N]$).
More precisely, we define the multi-species ASEP as an evolving bijection from $\Z$ or a finite discrete interval to itself, such that for any pair of nearest neighbors, if the left number is smaller than the right one, they swap with rate $1$; otherwise they swap with rate $q$; and such swaps happen independently for all pairs of nearest neighbors.

The mixing of the single-species ASEP on a finite interval was also studied in \cite{benjamini2005mixing}, where a `pre-cutoff' was proven.
Namely, it was shown that the total-variation distance decays from $1$ to $0$ in a certain time scale, but not necessarily in a small time window.
The cutoff problem (of the single-species ASEP) was settled in \cite{labbe2019cutoff} (together the biased card shuffling cutoff), using hydrodynamics of the ASEP on $\Z$.
In \cite{labbe2019cutoff} Labb{\'e} and Lacoin expected that the cutoff window should be in the order of $N^{1/3}$ with the cutoff profile related to the so-called Airy$_2$ process.
These were later established by Bufetov and Nejjar in \cite{bufetov2022cutoff}. They verified the $N^{1/3}$ order of cutoff window and showed that the cutoff profile is given by the GUE Tracy-Widom distribution (which is the one point distribution of the Airy$_2$ process, and was introduced by Tracy and Widom in \cite{tracy1994level} as the scaling limit of the fluctuation of the largest eigenvalue of the Gaussian Unitary Ensemble (GUE) matrices).
As noted in \cite[Section 1.2]{bufetov2022cutoff}, while proving cutoff for the biased card shuffling and the single-species ASEP follow similar and largely related arguments (in \cite{labbe2019cutoff}), it is `significantly more delicate' to study the cutoff profile of the biased card shuffling than the single-species ASEP. Due to reasons to be explained shortly, in \cite{bufetov2022cutoff} they also conjectured the $N^{1/3}$ cutoff window and the Tracy-Widom GOE cutoff profile for the biased card shuffling, and this conjecture is verified by our Theorem \ref{thm:main}.

We would like to mention that, these mixing time problems in the slightly different symmetric setting, i.e. where $q=1$ rather than $q\in [0,1)$, have also been studied.
The behavior of these chains (single-species/multi-species simple exclusion process in a finite interval) in this symmetric setting is quite different.
For both chains the mixing time is in the order of $N^2\log(N)$ \cite{wilson2004mixing}, and cutoff has been proven in \cite{lacoin2016mixing}; and for the single-species ASEP on a circle the cutoff window is proven to be in the order of $N^2$, with Gaussian cutoff profile \cite{lacoin2016cutoff}.
See also e.g. \cite{gantert2020mixing, gonccalves2021sharp, salez2022universality}, and the survey \cite{lacoin2021mixing} and the references therein, for some recent developments in this direction.\\

\noindent\textbf{The totally asymmetric case.} 
For the biased card shuffling, the special case where $q=0$ has been studied in the literature, first in \cite{angel2009oriented}, under the name `Oriented Swap Process' (OSP) and as a type of random sorting algorithm.
Note that in this case, there is a unique absorbing state $\lambda \in S_N$ where $\lambda(x)=N+1-x$ for each $x\in [N]$, and mixing of this Markov chain degenerates into the `absorbing time', i.e. the time of reaching this absorbing state.
In \cite{angel2009oriented}, the authors studied various aspects of the OSP, including the trajectories of the numbers, and various aspects of the absorbing time.
In particular, they showed that the absorbing time divided by $N$ converges to $2$ in probability. They also studied the finishing time of each number, i.e. the time a number $k\in[N]$ makes its last move.
This can be viewed as the finishing time (in a segment) of the \emph{totally asymmetric simple exclusion process} (TASEP), which is the degeneration of the (single-species) ASEP to the $q=0$ case.
In \cite{angel2009oriented} the authors showed that the one number finishing time has fluctuation in the order of $N^{1/3}$ with GUE Tracy-Widom scaling limit, using TASEP convergence results from \cite{Jo99}.

A question asked in \cite{angel2009oriented} is to figure out the order and distribution of the absorbing time (which can be viewed as the maximum of the finishing times over all $k\in [N]$).
This was settled in \cite{bufetov2022absorbing}, which proved that the absorbing time fluctuation is in the order of $N^{1/3}$ with GOE Tracy-Widom scaling limit, using recently proved symmetries of the six-vertex model \cite{borodin2019shift}. (This is also the reason why the biased card shuffling cutoff profile is conjectured to be GOE Tracy-Widom in \cite{bufetov2022cutoff}.)
Another direction pointed out in \cite{angel2009oriented} is to generate the results of OSP to the partially asymmetric (i.e. $q>0$) case.
Our Theorem \ref{thm:main} can be viewed as such a generalization of \cite{bufetov2022absorbing}.
We would also like to mention a recent interesting development \cite{aggarwal2022asep}, which likely leads to generating the results on trajectories in \cite{angel2009oriented} to the $q>0$ case.

\begin{center}
\begin{tabularx}{0.8\textwidth} { 
  | >{\centering\arraybackslash}X 
  | >{\centering\arraybackslash}X 
  | >{\centering\arraybackslash}X
  | >{\centering\arraybackslash}X | }
 \hline
 &  Mixing time and cutoff & (Single-species) exclusion process cutoff profile & Multi-species exclusion process cutoff profile \\
 \hline
Totally Asymmetric ($q=0$)  & \cite{angel2009oriented}  & \cite{angel2009oriented}\quad\quad\quad\quad\quad\quad   GUE Tracy-Widom  & \cite{bufetov2022absorbing}\quad\quad\quad\quad\quad\quad   GOE Tracy-Widom  \\
\hline
Partially Asymmetric ($q\in(0,1)$) & \cite{benjamini2005mixing}\cite{labbe2019cutoff} & \cite{bufetov2022cutoff}\quad\quad\quad\quad\quad   GUE Tracy-Widom  & Our Theorem \ref{thm:main}\quad\quad\quad\quad   GOE Tracy-Widom  \\
\hline
Symmetric ($q=1$) & \cite{wilson2004mixing}\cite{lacoin2016mixing}  & \cite{lacoin2016cutoff} (on the cycle) Gaussian & unknown \\
\hline
\end{tabularx}    
\captionof{table}{A summary of some results on the mixing/cutoff of single-species/multi-species exclusion processes.}\label{tab:1}

\end{center}

\subsection{Our contribution and main tools}
We next describe our strategy and the used tools, in comparison with related previous works (in particular, \cite{bufetov2022cutoff} and \cite{bufetov2022absorbing}).
As discussed above, \cite{bufetov2022cutoff} studies the general $q\in [0,1)$ problem, but at the level of a single-species ASEP projection; on the other hand, \cite{bufetov2022absorbing} studies the multi-species model, but in the special $q=0$ case.

Our proofs use ingredients from both papers.
As \cite{bufetov2022absorbing}, we can use a shift-invariance property of the multi-species ASEP (from \cite{borodin2019shift, galashin2021symmetries}), to reduce problems on the multi-species ASEP on $\Z$ to problems on the single-species ASEP on $\Z$ with step initial configuration.
In \cite{bufetov2022absorbing} classical convergence results on the TASEP were then used to get the GOE Tracy-Widom limit, while here we instead use the more recently proved ASEP convergence to the KPZ fixed point \cite{quastel2022convergence}.
Then the main task is to connect the cutoff profile problem in a finite interval, with the distribution of the multi-species ASEP on $\Z$.

In the $q=0$ case, due to that swaps happen only in one direction, there is a natural coupling between the multi-species TASEP on a finite interval and on $\Z$, using some truncation operators, as first described in \cite{angel2009oriented}.
Such relation is not available when $q>0$. 
Besides, the total-variation distance (in the $q>0$ case) is also much less tractable than the absorbing time (in the $q=0$ case).

Our main contribution is to develop new ideas to overcome these difficulties, and build the connection in the above `main task'.
We use the height function representation of the single-species ASEP, and Hecke algebra identities; and the core arguments are given in Section \ref{ssec:pmain} and Section \ref{sec:halg} below.
In more details, we couple the biased card shuffling (with a deterministic initial configuration) with the stationary one, so that the swaps are synchronized. This is the \emph{basic coupling} to be defined shortly in Section \ref{ssec:mset}.
Then the mixing time is reduced to the stopping time when they are equal for the first time.
By considering the single-species ASEP projections and using ordering properties of the ASEP height function, this problem is eventually reduced to comparing two (single-species) ASEPs on $\Z$ with different initial configurations.

For two (single-species) ASEPs under the basic coupling, such that the set of particle locations of one of them contains the set of particle locations of the other one, they are equivalent to the ASEP with \emph{second-class particles}.
A second-class particle is a particle which can swap with a hole next to it in the right, or a particle next to it in the left, with rate $1$ respectively.
It can also swap with a hole next to it in the left, or a particle next to it in the right, with rate $q$ respectively.
It is straightforward to check that, by replacing each second-class particle with a hole, or replacing each second-class particle with a particle, one can get two (single-species) ASEPs which are under the basic coupling.
Thus to study the difference of these two (single-species) ASEPs, one just need to track the locations of the second-class particles.
For this we use Hecke algebra identities, reducing it to events under the Mallows measure, which is well-understood with explicit distribution functions (e.g. \eqref{eq:maldis} and \eqref{eq:Pmnb} below).

Some of these ingredients, such as the basic coupling and Hecke algebra, have also appeared in \cite{bufetov2022cutoff}.
However, there are several key differences between the arguments in this paper and \cite{bufetov2022cutoff}.
The most significant aspect is in our usages of Hecke algebra.
In \cite{bufetov2022cutoff} Hecke algebra was used to translate events on certain single-species ASEPs to events on certain ASEPs with second-class particles, which were analyzed directly.
In this paper, Hecke algebra is used in the opposite direction.
Namely, we introduce second-class particles for we wish to compare two (single-species) ASEPs; and the Hecke algebra is used to reduce events on the ASEP with second-class particles to events on another (single-species) ASEP without second-class particles.
Besides, while ordering of particle configurations was also used in \cite{bufetov2022cutoff}, it is more extensively exploited in this paper, with the help of the more intuitive height function.
Using these we can directly compare ASEPs on a finite interval and $\Z$, and such comparison was not obtained in \cite{bufetov2022cutoff}.
Finally, we mention that our proof can be easily adapted to the single-species ASEP setting, to get an alternative proof of the main result of \cite{bufetov2022cutoff}
(see Remark \ref{rem:rec} below).

\subsection*{Notations and terminologies}
Some basic notations are used throughout this paper (and some of them have appeared already).
For any $x,y \in \R\cup\{-\infty, \infty\}$ we let $x\wedge y=\min\{x,y\}$ and $x\vee y=\max\{x,y\}$, and let $\llbracket x,y\rrbracket$ be the discrete interval, containing all integers $\ge x$ and $\le y$. We also write $[N]=\llbracket 1,N\rrbracket$ for any $N\in\N$.
Below $q$ always denotes the rate of reversed jumps or swaps, and we always assume that $q\in [0,1)$.

\subsection*{Organization of the remaining text}
The remaining text mainly focuses on the proofs.
In Sections \ref{ssec:mallows} and \ref{ssec:mset},  we set up some notations and provide some preliminary lemmas; then in Section \ref{ssec:pmain} we state the main steps of our arguments, as three propositions, and prove Theorem \ref{thm:main} assuming them.
The most important one (of the three steps), where we compare the biased card shuffling and the multi-species ASEP on $\Z$, is implemented in Section \ref{sec:halg}.
In the last two sections we do the remaining steps respectively. Specifically, in Section \ref{sec:hit} we prove an estimate for certain `ground state hitting time' of the single-species ASEP, and in Section \ref{sec:asepk} we deduce convergence to the GOE Tracy-Widom distribution using known ASEP shift-invariance and convergence results.

\section{Preliminaries and main steps}   \label{sec:mains}
\subsection{The stationary Mallows measure}  \label{ssec:mallows}
We start by giving the formal definition of the Mallows measures, which are stationary measures of the biased card shuffling or the multi-species ASEP on a finite interval.
For the convenience of later arguments, we consider the interval $\llbracket m,n\rrbracket$ for integers $m\le n$.
As in \cite{bufetov2022cutoff}, we let $S_{m,n}$ denote the set of all bijections between $\llbracket m,n\rrbracket$ and itself,
and let $\cM_{m,n}$ denote the Mallows measure on $S_{m,n}$, such that for any $w\in S_{m,n}$,
\begin{equation}  \label{eq:maldis}
\cM_{m,n}(w)=q^{\kappa(w)}Z_{m,n},    
\end{equation}
where (and also for the rest of this paper) we denote $\kappa(w)=\sum_{m\le i<j\le n}\don[w(i)<w(j)]$ as the `energy' for any $w\in S_{m,n}$, and
\begin{equation}  \label{eq:defz}
Z_{m,n}=\frac{1}{\sum_{w\in S_{m,n}} q^{\kappa(w)} }=\prod_{i=1}^{n-m+1}\frac{1-q}{1-q^i}.    
\end{equation}
Note that there is a unique $w\in S_{m,n}$ with $\kappa(w)=0$, i.e. $w:i=\mapsto m+n-i$; and that is the `ground state'.
It is straightforward to check that $\cM_{m,n}$ is the stationary measure of the biased card shuffling on $\llbracket m,n\rrbracket$.

For each $k\in \llbracket 0,n-m+1\rrbracket$, we let $\cP_{m,n}^k$ be the measure on $\{0,1\}^{\llbracket m,n\rrbracket}$, obtained from $\cM_{m,n}$ under the projection where each $w\in S_{m,n}$ is mapped to $x\mapsto \don[w(x)\le m+k-1]$.
In words, under $\cP_{m,n}^k$ there are $k$ particles and $n-m+1-k$ holes, if we view each $1$ as a particle and each $0$ as a hole.
One can readily check that, for any $\omega\in\{0,1\}^{\llbracket m,n\rrbracket}$, 
\begin{equation}   \label{eq:Pmnb}
\cP_{m,n}^k(\omega)=\don\big[\sum_{i\in\llbracket m,n\rrbracket} \omega(i)=k\big] q^{\kappa(\omega)} Z_{m,n}^k,    
\end{equation}
where (as a slight misuse of notations, and also for the rest of this paper) $\kappa(\omega)=\sum_{m\le i<j\le n}\don[\omega(i)=1]\don[\omega(j)=0]$ is the `energy' for any $\omega \in\{0,1\}^{\llbracket m,n\rrbracket}$, and
\[
Z_{m,n}^k=\frac{1}{\sum_{w\in \{0,1\}^{\llbracket m,n\rrbracket}, \sum_{i\in\llbracket m,n\rrbracket} \omega(i)=k} q^{\kappa(w)} }=\frac{\prod_{i=1}^{k}(1-q^i)\prod_{i=1}^{n-m+1-k}(1-q^i)}{\prod_{i=1}^{n-m+1}(1-q^i)}.
\]
Note that here the `ground state' is given by $\omega:i\mapsto \don[i\ge n-k+1]$.
From this definition of $\cP_{m,n}^k$ as a projection of $\cM_{m,n}$, it is stationary under the (single-species) ASEP evolution on $\llbracket m,n\rrbracket$.

We will frequently use the following estimate of the `energy' under $\cP_{m,n}^k$.
\begin{lemma}  \label{lem:invbd}
There is a constant $C>0$ depending only on $q$, such that for any integers $m\le n$, $k\in \llbracket 0,n-m+1\rrbracket$, and any $a\in\N$,
\begin{equation}  \label{eq:invbd}
\cP_{m,n}^k\big(\{\omega\in \{0,1\}^{\llbracket m,n\rrbracket}: \kappa(\omega) > a\}\big)<Cq^{a/2}.    
\end{equation}
\end{lemma}
\begin{proof}
Let $\sum_0^\infty \alpha_i z^i$ be the Taylor series for $\prod_{i=1}^\infty(1-z^i)^{-1}$.
By \eqref{eq:Pmnb}, and noting that $Z_{m,n}^k\le 1$, the left-hand side of \eqref{eq:invbd} is at most
\[
\sum_{i=a+1}^\infty \alpha_i q^i<q^{(a+1)/2}\prod_{i=1}^\infty(1-q^{i/2})^{-1}.
\]
By taking $C=q^{1/2}\prod_{i=1}^\infty(1-q^{i/2})^{-1}$ the conclusion follows.
\end{proof}
For $N\in\N$ and $k\in\llbracket 0,N\rrbracket$, below we also write $S_N=S_{1,N}$, $\cM_N=\cM_{1,N}$, and $\cP_N^k=\cP_{1,N}^k$, for simplicity of notations.

\subsection{Setup for several processes}   \label{ssec:mset}

In this subsection we define some processes to be used in the proof of Theorem \ref{thm:main}, and discuss some basic properties of them.
\begin{itemize}
    \item Let $\bzeta=(\zeta_t)_{t\ge 0}$ be the multi-species ASEP, such that the configuration at any time $t\ge 0$ is denoted as a bijection $\zeta_t:\Z\to\Z$.
    Let the initial configuration $\zeta_0$ be the identity map of $\Z$.
\end{itemize}
In defining the following processes and for the rest of this paper, we take $N\in\N$.
\begin{itemize}
    \item We let $\bxi=(\xi_t)_{t\ge 0}$ be the biased card shuffling of size $N$, with $\xi_t:[N]\to [N]$ being the state at time $t$. Let the initial configuration $\xi_0$ be the identity map of $[N]$.
    \item We let $\bla=(\lambda_t)_{t\ge 0}$ be the same as $\bxi=(\xi_t)_{t\ge 0}$, except for that the initial configuration $\lambda_0$ is some general deterministic element of $S_N$.
    \item Let $\obxi=(\oxi_t)_{t\ge 0}$ be the stationary biased card shuffling of size $N$.
    Then for any $t\ge 0$ the law of $\oxi_t$ is $\cM_N$.
\end{itemize}
\noindent\textbf{Basic coupling.}
For the multi-species or single-species ASEP on $\Z$ or a discrete interval, the evolution can be generated by two independent Poisson point processes on $\Z\times \R_{\ge 0}$, with rates $1$ and $q$ respectively.
In the setting of the multi-species ASEP on $\Z$, if there is a point at $(x,t)$ in the rate $1$ Poisson point process, at time $t$ we swap the numbers at $x$ and $x+1$, if before the swap the number at $x$ is smaller than the number at $x+1$; and if there is a point at $(x,t)$ in the rate $q$ Poisson point process, at time $t$ we swap the numbers at $x$ and $x+1$, if before the swap the number at $x$ is larger than the number at $x+1$. 
In the discrete interval setting, one only considers such points $(x,t)$ with both $x$ and $x+1$ in the discrete interval.
One can generate the single-species ASEP (from these Poisson point processes) similarly. 
The \emph{basic coupling} between two or more multi-species or single-species ASEPs is the coupling under which all these processes are generated from the same pair of Poisson point processes.

For the defined processes $\bzeta$, $\bxi$, $\bla$, and $\obxi$, we couple all of them with the basic coupling.\\

\noindent\textbf{Projections.}  For any $k\in \llbracket 0,N\rrbracket$, we define the processes $\bxi^k=(\xi^k_t)_{t\ge 0}$, $\bzeta^k=(\zeta^k_t)_{t\ge 0}$, and $\obxi^k=(\oxi^k_t)_{t\ge 0}$, as follows.
For each $t\ge 0$, we let $\xi_t^k, \lambda_t^k, \oxi_t^k:[N]\to\{0,1\}$ be the projections, where
\[
\xi_t^k(x) = \don[\xi_t(x)\le k],\quad
\lambda_t^k(x) = \don[\lambda_t(x)\le k],\quad
\oxi_t^k(x) = \don[\oxi_t(x)\le k].
\]
And for each $k\in\Z$ and $t\ge 0$, we let $\zeta_t^k:\Z\to\{0,1\}$ be the projection where
\[
\zeta_t^k(x) = \don[\zeta_t(x)\le k].
\]
Then these projections are single-species ASEPs on $\Z$ or $[N]$, and are also coupled together under the basic coupling (i.e. generated with the same Poisson point processes).
We note that, for each $t\ge 0$, given $\zeta_t^k$ for all $k\in\Z$, one can recover $\zeta_t$ by letting
\[
\zeta_t(x)=\min\{k:\zeta_t^k(x)=1\}.
\]
Similar recovery statements hold for $\xi_t$, $\lambda_t$, and $\oxi_t$.
\\

\noindent\textbf{Height function.} For any $\omega:\Z\to \{0,1\}$ such that $\omega(x)$ stabilizes as $x\to-\infty$, we define the height function $h\{\omega\}:\Z\to \Z$, as follows.
\begin{enumerate}
    \item For any $x\in \Z$, let $h\{\omega\}(x)-h\{\omega\}(x-1)=2\omega(x)-1$.
    \item If $\omega(x)=0$ for all small enough, we let $h\{\omega\}(x)=x$ for all small enough $x$.
    \item If $\omega(x)=1$ for all small enough, we let $h\{\omega\}(x)=-x$ for all small enough $x$.    
\end{enumerate}
A quick and useful observation is that $h\{\omega\}$ is $1$-Lipschitz.
For any integers $m<n$ and $\omega:\llbracket m, n\rrbracket\to \{0,1\}$, we also define $h\{\omega\}$ by assuming $\omega(x)=0$ for any $x\in \Z$, $x<m$ and $\omega(x)=1$ for any $x\in\llbracket n,\infty\rrbracket$ (see Figure \ref{fig:height}).

Such height functions are widely used in the study of interacting particle systems. 
For $\omega$ evolving as the single-species ASEP, the height function $h\{\omega\}$ evolves in the following way: for some $x\in\Z$,
if $h\{\omega\}(x-1)=h\{\omega\}(x+1)=h\{\omega\}(x)+1$ for some $x\in\Z$, then $h\{\omega\}(x)$ changes to $h\{\omega\}(x)+2$ with rate $1$; if $h\{\omega\}(x-1)=h\{\omega\}(x+1)=h\{\omega\}(x)-1$, then $h\{\omega\}(x)$ changes to $h\{\omega\}(x)-2$ with rate $q$.
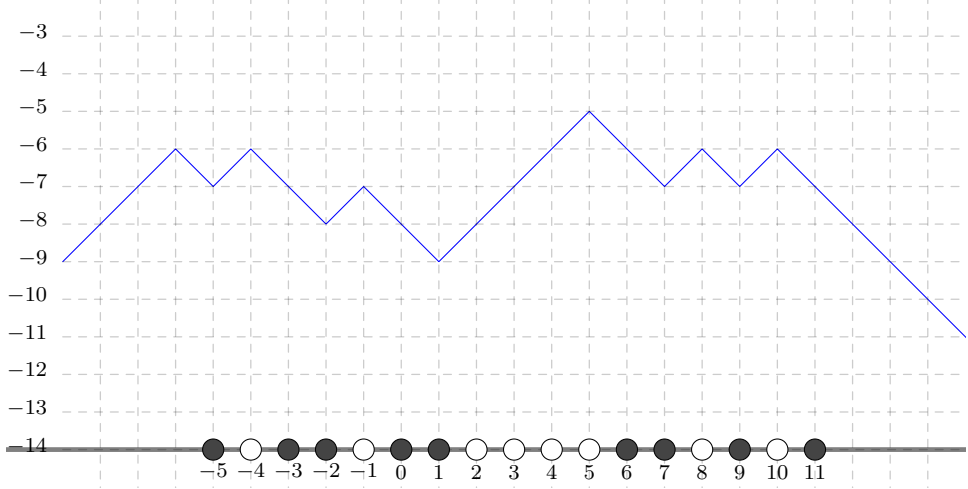
\begin{figure}[hbt!]
    \centering
\begin{tikzpicture}[line cap=round,line join=round,>=triangle 45,x=5cm,y=5cm]
\clip(-2.95,-0.2) rectangle (-0.4,1.1);

\draw [blue] plot coordinates {(-2.8,0.4) (-2.5,0.7) (-2.4,0.6) (-2.3,0.7) (-2.1,0.5) (-2,0.6) (-1.8,0.4) (-1.4,0.8) (-1.2,0.6) (-1.1,0.7) (-1,0.6) (-0.9,0.7) (0,-0.2)};

\foreach \i in {3,...,25}
{
\draw [line width=0.5pt, opacity=0.2] [dashed] (\i/10-3,-0.6) -- (\i/10-3,1.6);
}
\foreach \i in {-14,...,-3}
{
\draw [line width=0.5pt, opacity=0.2] [dashed] (-2.8,\i/10+1.3) -- (1,\i/10+1.3);
\begin{scriptsize}
\draw (-2.82,\i/10+1.31) node[anchor=east]{$\i$};
\end{scriptsize}
}

\draw [line width=1.8pt, opacity=0.5] (-3,-0.1) -- (-0.2,-0.1);
\draw [fill=uuuuuu] (-2.4,-0.1) circle (4.0pt);
\draw [fill=white] (-2.3,-0.1) circle (4.0pt);
\draw [fill=uuuuuu] (-2.2,-0.1) circle (4.0pt);
\draw [fill=uuuuuu] (-2.1,-0.1) circle (4.0pt);
\draw [fill=white] (-2.,-0.1) circle (4.0pt);
\draw [fill=uuuuuu] (-1.9,-0.1) circle (4.0pt);
\draw [fill=uuuuuu] (-1.8,-0.1) circle (4.0pt);
\draw [fill=white] (-1.7,-0.1) circle (4.0pt);
\draw [fill=white] (-1.6,-0.1) circle (4.0pt);
\draw [fill=white] (-1.5,-0.1) circle (4.0pt);
\draw [fill=white] (-1.4,-0.1) circle (4.0pt);
\draw [fill=uuuuuu] (-1.3,-0.1) circle (4.0pt);
\draw [fill=uuuuuu] (-1.2,-0.1) circle (4.0pt);
\draw [fill=white] (-1.1,-0.1) circle (4.0pt);
\draw [fill=uuuuuu] (-1.,-0.1) circle (4.0pt);
\draw [fill=white] (-0.9,-0.1) circle (4.0pt);
\draw [fill=uuuuuu] (-0.8,-0.1) circle (4.0pt);

\begin{scriptsize}
\draw (-2.4,-0.2) node[anchor=south]{$-5$};
\draw (-2.3,-0.2) node[anchor=south]{$-4$};
\draw (-2.2,-0.2) node[anchor=south]{$-3$};
\draw (-2.1,-0.2) node[anchor=south]{$-2$};
\draw (-2.,-0.2) node[anchor=south]{$-1$};
\draw (-1.9,-0.2) node[anchor=south]{$0$};
\draw (-1.8,-0.2) node[anchor=south]{$1$};
\draw (-1.7,-0.2) node[anchor=south]{$2$};
\draw (-1.6,-0.2) node[anchor=south]{$3$};
\draw (-1.5,-0.2) node[anchor=south]{$4$};
\draw (-1.4,-0.2) node[anchor=south]{$5$};
\draw (-1.3,-0.2) node[anchor=south]{$6$};
\draw (-1.2,-0.2) node[anchor=south]{$7$};
\draw (-1.1,-0.2) node[anchor=south]{$8$};
\draw (-1.,-0.2) node[anchor=south]{$9$};
\draw (-.9,-0.2) node[anchor=south]{$10$};
\draw (-.8,-0.2) node[anchor=south]{$11$};
\end{scriptsize}

\end{tikzpicture}
\caption{
An illustration of the height function for an single-species ASEP configuration in a finite interval.
}  
\label{fig:height}
\end{figure}

We can define a partial ordering for these height functions.
For any $f,g:\Z\to\R$, we write $f\le g$ if $f(x)\le g(x)$ for any $x\in \Z$.
One reason of using these height functions is that, as one can check, this ordering is preserved under the ASEP evolution with the basic coupling.
For example, for the processes defined above, the following orderings hold.
\begin{lemma}  \label{lem:hei-ord1}
We have $h\{\xi_t^k\}\le h\{\lambda_t^k\}$, $h\{\xi_t^k\}\le h\{\zeta_t^k\}$, and $h\{\xi_t^k\}\le h\{\oxi_t^k\}$, for any $k\in\llbracket 0,N\rrbracket$ and $t\ge 0$.
\end{lemma}
\begin{proof}
We first consider the case where $t=0$.
We have \[h\{\xi_0^k\}(0)=h\{\lambda_0^k\}(0)=h\{\zeta_0^k\}(0)=h\{\oxi_0^k\}(0)=0,\] and \[h\{\xi_0^k\}(N)=h\{\lambda_0^k\}(N)=h\{\zeta_0^k\}(N)=h\{\oxi_0^k\}(N)=N-2k,\]
since
\[\sum_{x\in[N]}\xi_0^k(x)=\sum_{x\in[N]}\lambda_0^k(x)=\sum_{x\in[N]}\zeta_0^k(x)=\sum_{x\in[N]}\oxi_0^k(x)=k.\]
Then since $h\{\xi_0^k\}$ has slope $1$ in $\llbracket-\infty, 0\rrbracket\cup\llbracket k,N\rrbracket$ and slope $-1$ in $\llbracket 0,k\rrbracket\cup\llbracket N,\infty\rrbracket$, and that $h\{\lambda_0^k\}$, $h\{\zeta_0^k\}$, and $h\{\oxi_0^k\}$ are $1$-Lipschitz, we have that these inequalities hold when $t=0$.

For any $t>0$, first note that the inequalities always hold in $\llbracket -\infty, 0\rrbracket$ and $\llbracket N,\infty\rrbracket$.
Indeed, we have $h\{\xi_t^k\}(x)=h\{\lambda_t^k\}(x)=h\{\oxi_0^k\}(x)=x$ for $x\in\llbracket -\infty, 0\rrbracket$, and $h\{\xi_t^k\}(x)=h\{\lambda_t^k\}(x)=h\{\oxi_0^k\}(x)=2N-2k-x$ for $x\in\llbracket N,\infty\rrbracket$; and $h\{\zeta_t^k\}(x)\ge (-x)\vee (x-2k)$ for all $x\in\Z$ (since $h\{\zeta_t^k\}(x)=-x$ for all small enough $x$ and $h\{\zeta_t^k\}(x)=x-2k$ for all large enough $x$, and $h\{\zeta_t^k\}$ is $1$-Lipschitz). 
In the interval $\llbracket 1, N-1\rrbracket$, from the above description on how these height functions evolve in time, it is straightforward to check that, if the inequalities hold at any time, any nearest neighbor swap in the basic coupling would not break them.
Then the inequalities hold for any $t>0$.
\end{proof}

\subsection{Proof of the main result}  \label{ssec:pmain}

With the processes defined in the previous subsection, we need (1) to upper bound the first time $T$ when $\lambda_T=\oxi_T$; (2) to lower bound the first time $T$ when $\oxi_T=\xi_T$, and show that $\oxi_t$ and $\xi_t$ are `quite different' for $t$ between $T-o(N^{1/3})$ and $T$.
We now give the main steps to accomplish these tasks.

We start by explaining how the Tracy-Widom GOE distribution is involved.
This distribution appears in scaling limits of the single-species ASEP, as we will later see using the single-species ASEP convergence results from \cite{quastel2022convergence}.
From these results and using an ASEP shift-invariance properties proved in \cite{borodin2019shift, galashin2021symmetries}, we can deduce convergence results for the multi-species ASEP on $\Z$, as follows.

For $b\in\R$ and $t\ge 0$, let $\cD_{N,t}^{b}$ denote the event where $h\{\zeta_t^k\}(N-k)>N-k+b$ for any $k\in \llbracket 0,N\rrbracket$.
\begin{prop}   \label{prop:asepi}
For any $\tau\in \R$, and any real number sequence $\{b_N\}_{N\in\N}$ such that $\lim_{N\to\infty} N^{-1/3}b_N\to 0$, we have
\[
\lim_{N\to\infty}\PP[\cD_{N,2(1-q)^{-1}(N+\tau N^{1/3})}^{b_N}]=F_{\GOE}(2^{2/3}\tau).
\]
\end{prop}
The proof of this proposition will be given in Section \ref{sec:asepk}.

In the TASEP setting (i.e. when $q=0$), the multi-species process on a finite interval can be obtained from the multi-species process on $\Z$, using truncation operators (see e.g. \cite{angel2009oriented}). Thus Proposition \ref{prop:asepi} directly implies that the absorbing time has GOE Tracy-Widom distribution asymptotically; and that is the proof in \cite{bufetov2022absorbing}.

For the general $q>0$ setting, we need to study the more delicate mixing time (rather than absorbing time).
A more important difference is that, there is no such truncation operators and no exact coupling structure between the process on $\Z$ and the process on a finite interval.
Thus we instead show that, under the basic coupling, the processes on $\Z$ and a finite interval are `close to each other' in terms of the height functions of their projections, in the following sense.

For $a\in\N$, $b\in\R$, $k\in\llbracket 0,N\rrbracket$, and $t\ge 0$, let $\cA_{N,k,t}^{a,b}$ denote the event where $h\{\zeta_t^k\}(N-k)>N-k+b$, and 
\[\text{either  }\;
h\{\xi_t^k\}(N-k-a)<N-k-a\;\text{  or  }\; h\{\xi_t^k\}(N-k+a)<N-k-a.\]
\begin{prop}  \label{prop:hec}
There is a constant $C>0$ depending only on $q$, such that
$\PP[\cA_{N,k,t}^{a,b}]<Cq^{a/2}$ when $b>4a$.
\end{prop}
We will prove this proposition in Section \ref{sec:halg}, using Hecke algebra techniques.

In analyzing the first time $T$ when $\lambda_T=\oxi_T$, a key ingredient we need is that, if $\lambda_t$ and $\oxi_t$ are `close to each other' at some time $t$, they will be the same in a short amount of time. We state this result as follows.

For $a\in\N$, $r>0$, $k\in\llbracket 0,N\rrbracket$, and $t\ge 0$, let $\cB_{N,k,t}^{a,r}$ denote the event where
\[
h\{\xi_{t+i}^k\}(N-k-a)=h\{\xi_{t+i}^k\}(N-k+a)=N-k-a,\quad \forall i\in \llbracket 0, r \rrbracket,\]
while $h\{\xi_{t+s}^k\}(N-k)<N-k$ for each $s\in [0, r]$ (in other words, $\xi_{t+s}^k$ is not at the `ground state' for each $s\in [0, r]$, since $h\{\xi_{t+s}^k\}(N-k)=N-k$ is equivalent to that $\xi_{t+s}^k(x)=\don[x\ge N-k+1]$).
\begin{prop}    \label{prop:hitrp}
There is a constant $c>0$ depending only on $q$, such that
$\PP[\cB_{N,k,t}^{a,\log(N)^4}]<e^{-c\log(N)^{2}}$ when $2\le a\le\log(N)^2$.
\end{prop}
This proposition will be proved in Section \ref{sec:hit}, using an ASEP hitting time bound from \cite{benjamini2005mixing}.

We now give the proof of our main result (Theorem \ref{thm:main}), which mainly consists of assembling these propositions.

\begin{proof}[Proof of Theorem \ref{thm:main}: upper bound]
The key thing is to show that, for whatever initial configuration $\lambda_0$, we have
\begin{equation}  \label{eq:lacov}
\liminf_{N\to\infty}\PP[\lambda_{2(1-q)^{-1}(N+\tau N^{1/3})+\log(N)^4}=\oxi_{2(1-q)^{-1}(N+\tau N^{1/3})+\log(N)^4}] \ge F_{\GOE}(2^{2/3}\tau).   
\end{equation}
Let's first explain the strategy to prove it. In light of Proposition \ref{prop:asepi}, the main idea is to show that the event in the left-hand side of \eqref{eq:lacov} is roughly implied by $\cD_{N,2(1-q)^{-1}(N+\tau N^{1/3})}^{8\log(N)^4}$.
We will show that, under $\cD_{N,2(1-q)^{-1}(N+\tau N^{1/3})}^{8\log(N)^4}$, for each $k\in\llbracket 0,N\rrbracket$, $\xi_{2(1-q)^{-1}(N+\tau N^{1/3})}^k$ is `close' to the ground state, i.e., most locations in $\llbracket N-k+1, N\rrbracket$ have particles, and most locations in $\llbracket 1, N-k\rrbracket$ have holes.
Then there exists some small $s>0$, such that $\xi_{2(1-q)^{-1}(N+\tau N^{1/3})+s}^k$ is precisely at the ground state.
These two steps are given by Propositions \ref{prop:hec} and \ref{prop:hitrp}.
Then $h\{\xi_{2(1-q)^{-1}(N+\tau N^{1/3})+s}^k\}$ reaches the `largest possible state', i.e. it has slope $1$ in $\llbracket -\infty, n-k\rrbracket$ and slope $-1$ in $\llbracket n-k, \infty\rrbracket$.
By ordering of the height functions (Lemma \ref{lem:hei-ord1}), this implies the event in the left-hand side of \eqref{eq:lacov}.

We now give the details of proving \eqref{eq:lacov}. Consider the event
\begin{multline*}
\cE_0=\cD_{N,2(1-q)^{-1}(N+\tau N^{1/3})}^{8\log(N)^4} \setminus \\
\cE_1\cup\Big(\bigcup_{k\in\llbracket 0,N\rrbracket, i\in \llbracket 0, \log(N)^4\rrbracket} \cA_{N,k,2(1-q)^{-1}(N+\tau N^{1/3})+i}^{\lfloor \log(N)^2 \rfloor,4\log(N)^4}\Big) \cup \Big(\bigcup_{k\in\llbracket 0,N\rrbracket} \cB_{N,k,2(1-q)^{-1}(N+\tau N^{1/3})}^{\lfloor \log(N)^2 \rfloor,\log(N)^4}\Big).    
\end{multline*}
Here $\cE_1$ is the following event:
\begin{itemize}
    \item For some $k\in \llbracket 0,N\rrbracket$ and $i\in \llbracket 0, \log(N)^4\rrbracket$, we have 
    \begin{equation}  \label{eq:ce1d1}
    h\{\zeta_{2(1-q)^{-1}(N+\tau N^{1/3})}^k\}(N-k)>N-k+8\log(N)^4,
    \end{equation} and 
    \begin{equation}  \label{eq:ce1d2}
    h\{\zeta_{2(1-q)^{-1}(N+\tau N^{1/3})+i}^k\}(N-k)\le N-k+4\log(N)^4.
    \end{equation}
\end{itemize}

Now let's assume that $\cE_0$ happens.
By $\cD_{N,2(1-q)^{-1}(N+\tau N^{1/3})}^{8\log(N)^4} \setminus \cE_1$, we have that for any $k\in \llbracket 0,N\rrbracket$ and $i\in \llbracket 0, \log(N)^4\rrbracket$, \eqref{eq:ce1d2} does not hold.
Then since $\cA_{N,k,2(1-q)^{-1}(N+\tau N^{1/3})+i}^{\lfloor \log(N)^2 \rfloor,4\log(N)^4}$ does not hold, we must have that
\begin{multline*}
h\{\xi_{2(1-q)^{-1}(N+\tau N^{1/3})+i}^k\}(N-k-\lfloor \log(N)^2 \rfloor)=h\{\xi_{2(1-q)^{-1}(N+\tau N^{1/3})+i}^k\}(N-k+\lfloor \log(N)^2 \rfloor)\\=N-k-\lfloor \log(N)^2 \rfloor.    
\end{multline*}
Using this for all $k\in \llbracket 0,N\rrbracket$ and $i\in \llbracket 0, \log(N)^4\rrbracket$, and that $\cB_{N,k,2(1-q)^{-1}(N+\tau N^{1/3})}^{\lfloor \log(N)^2 \rfloor,\log(N)^4}$ does not hold, we have that $\cE_2$ holds, with $\cE_2$ being the event:
\begin{itemize}
    \item For any $k\in \llbracket 0,N\rrbracket$, we have $h\{\xi_{2(1-q)^{-1}(N+\tau N^{1/3})+s_k}^k\}(N-k)=N-k$ for some $s_k\in [0, \log(N)^4]$.
\end{itemize}
In summary, we have shown that $\cE_0$ implies $\cE_2$.

We next show that under $\cE_2$, we must have $\oxi_{2(1-q)^{-1}(N+\tau N^{1/3})+\log(N)^4}=\lambda_{2(1-q)^{-1}(N+\tau N^{1/3})+\log(N)^4}$.
Indeed, assuming $\cE_2$, for any $k\in \llbracket 0,N\rrbracket$ we have that $h\{\xi_{2(1-q)^{-1}(N+\tau N^{1/3})+s_k}^k\}$ is the function with slope $1$ in $\llbracket -\infty, n-k\rrbracket$, and slope $-1$ in $\llbracket n-k, \infty\rrbracket$.
Thus by Lemma \ref{lem:hei-ord1}, the same is true for $h\{\oxi_{2(1-q)^{-1}(N+\tau N^{1/3})+s_k}^k\}$ and $h\{\lambda_{2(1-q)^{-1}(N+\tau N^{1/3})+s_k}^k\}$.
So we have
\[
\xi_{2(1-q)^{-1}(N+\tau N^{1/3})+s_k}^k=\oxi_{2(1-q)^{-1}(N+\tau N^{1/3})+s_k}^k=\lambda_{2(1-q)^{-1}(N+\tau N^{1/3})+s_k}^k.
\]
Since $\obxi$ and $\bla$ are under the basic coupling and $s_k\le \log(N)^4$, this implies that 
\[\oxi_{2(1-q)^{-1}(N+\tau N^{1/3})+\log(N)^4}^k=\lambda_{2(1-q)^{-1}(N+\tau N^{1/3})+\log(N)^4}^k.\]
As this holds for each $k\in \llbracket 0,N\rrbracket$, we have $\oxi_{2(1-q)^{-1}(N+\tau N^{1/3})+\log(N)^4}=\lambda_{2(1-q)^{-1}(N+\tau N^{1/3})+\log(N)^4}$.

Now we have
\begin{multline}  \label{eq:ineqce}
\PP[\oxi_{2(1-q)^{-1}(N+\tau N^{1/3})+\log(N)^4}=\lambda_{2(1-q)^{-1}(N+\tau N^{1/3})+\log(N)^4}] \ge \PP[\cE_2] \ge \PP[\cD_{N,2(1-q)^{-1}(N+\tau N^{1/3})}^{8\log(N)^4}] \\
-\PP[\cE_1]-\sum_{k\in\llbracket 0,N\rrbracket, i\in \llbracket 0, \log(N)^4\rrbracket} \PP[\cA_{N,k,2(1-q)^{-1}(N+\tau N^{1/3})+i}^{\lfloor \log(N)^2 \rfloor,4\log(N)^4}]
-\sum_{k\in\llbracket 0,N\rrbracket} \PP[\cB_{N,k,2(1-q)^{-1}(N+\tau N^{1/3})}^{\lfloor \log(N)^2 \rfloor,\log(N)^4}].
\end{multline}
To get \eqref{eq:lacov}, we next bound $\PP[\cE_1]$.

Let $T_0=0$. For each $j\in\N$, we let $T_j$ be the smallest positive number with
\[
h\{\zeta_{2(1-q)^{-1}(N+\tau N^{1/3})+T_j}^k\}(N-k)=h\{\zeta_{2(1-q)^{-1}(N+\tau N^{1/3})}^k\}(N-k)-2j;
\]
or $T_j=\infty$ if no such positive number exists.
Then given $T_{j-1}$ and $\zeta_{2(1-q)^{-1}(N+\tau N^{1/3})+T_{j-1}}^k$, $T_j-T_{j-1}$ stochastically dominates an $\Exp(q)$ random variable.
Thus $T_{\lfloor2\log(N)^4\rfloor}$ stochastically dominates the sum of $\lfloor2\log(N)^4\rfloor$ independent $\Exp(q)$ random variables.
On the other hand, if \eqref{eq:ce1d1} and \eqref{eq:ce1d2} happen for some $i\in \llbracket 0, \log(N)^4\rrbracket$, we must have $T_{\lfloor2\log(N)^4\rfloor}<\log(N)^4$.
Using a Chernoff bound, this happens with probability $<Ce^{-c\log(N)^4}$, for $c,C>0$ being universal constants.
By a union bound over $k$, we have that $\PP[\cE_1]<C(N+1)e^{-c\log(N)^4}$.

Using this bound of $\PP[\cE_1]$, and \eqref{eq:ineqce}, and Propositions \ref{prop:asepi}, \ref{prop:hec}, \ref{prop:hitrp}, and sending $N\to\infty$, we have that \eqref{eq:lacov} holds.

Finally, as the distributions of $\oxi_{2(1-q)^{-1}(N+\tau N^{1/3})+\log(N)^4}$ and $\lambda_{2(1-q)^{-1}(N+\tau N^{1/3})+\log(N)^4}$ are given by $\cM_N$ and $\cW_{N,2(1-q)^{-1}(N+\tau N^{1/3})+\log(N)^4}^{\lambda_0}$, respectively, by \eqref{eq:lacov} we have that
\[
\limsup_{N\to\infty}\max_{\lambda\in S_N} \|\cW_{N,2(1-q)^{-1}(N+\tau N^{1/3})+\log(N)^4}^\lambda-\cM_N\|_{\TV} \le 1-F_{\GOE}(2^{2/3}\tau).
\]
Since $t\mapsto \|\cW_{N,t}^\lambda-\cM_N\|_{\TV}$ is non-increasing, and that $F_{\GOE}$ is a continuous function (which is evident from \eqref{eq:goedef}), the upper bound of Theorem \ref{thm:main} follows.
\end{proof}
The lower bound proof is similar and simpler, since for this we just need to show that $\xi_{2(1-q)^{-1}(N+\tau N^{1/3})}^k$ and $\oxi_{2(1-q)^{-1}(N+\tau N^{1/3})}^k$ are different with probability at least (roughly) $1-F_{\GOE}(2^{2/3}\tau)$.
Only the ASEP on $\Z$ convergence (Proposition \ref{prop:asepi}) and ordering of the height functions (Lemma \ref{lem:hei-ord1}) will be used.
\begin{proof}[Proof of Theorem \ref{thm:main}: lower bound]
We consider the events $\cE$:
\[
h\{\xi_{2(1-q)^{-1}(N+\tau N^{1/3})}^k\}(N-k)\le N-k-\log(N)^2,\quad \exists k\in \llbracket 0,N\rrbracket.
\]
and $\ocE$:
\[
h\{\oxi_{2(1-q)^{-1}(N+\tau N^{1/3})}^k\}(N-k)\le N-k-\log(N)^2,\quad \exists k\in \llbracket 0,N\rrbracket.
\]
Since the distribution of $\xi_{2(1-q)^{-1}(N+\tau N^{1/3})}$ is $\cW_{N,2(1-q)^{-1}(N+\tau N^{1/3})}^{\xi_0}$ and the distribution of $\oxi_{2(1-q)^{-1}(N+\tau N^{1/3})}$ is $\cM_N$, we have
\[
\|\cW_{N,2(1-q)^{-1}(N+\tau N^{1/3})}^{\xi_0}-\cM_N\|_{\TV} \ge \PP[\cE]-\PP[\ocE].
\]
By Proposition \ref{prop:asepi} and Lemma \ref{lem:hei-ord1}, we have $\liminf_{N\to\infty}\PP[\cE]\ge 1-F_{\GOE}(2^{2/3}\tau)$.
It now suffices to show that $\PP[\ocE]\to 0$ as $N\to\infty$.

For each $k\in\llbracket 0,N\rrbracket$, the event $h\{\oxi_{2(1-q)^{-1}(N+\tau N^{1/3})}^k\}(N-k)\le N-k-\log(N)^2$ implies that
\[
h\{\oxi_{2(1-q)^{-1}(N+\tau N^{1/3})}^k\}(N-k-\lfloor \log(N)^2/2 \rfloor+1)\le N-k-\log(N)^2+\lfloor \log(N)^2/2 \rfloor-1.
\]
Then there must exist an integer $1\le x\le N-k-\lfloor \log(N)^2/2 \rfloor+1$ with $\oxi_{2(1-q)^{-1}(N+\tau N^{1/3})}^k(x)=0$.
This implies that
\[
\kappa(\oxi_{2(1-q)^{-1}(N+\tau N^{1/3})}^k)\ge \lfloor \log(N)^2/2 \rfloor,
\]
where we recall (from Section \ref{ssec:mallows}) that $\kappa(\oxi_{2(1-q)^{-1}(N+\tau N^{1/3})}^k)$ is the `energy'.
By Lemma \ref{lem:invbd}, the probability of this event is bounded by $Cq^{\log(N)^2/4}$, for some $C>0$ depending only on $q$.
Then by taking a union bound over $k$ we have that $\lim_{N\to\infty}\PP[\ocE]=0$, and we finish this proof.
\end{proof}
\begin{rem}  \label{rem:rec}
Using Propositions \ref{prop:hec} and \ref{prop:hitrp} for a single $k$, and known one-point distribution convergence of ASEP to the GUE Tracy-Widom distribution (e.g. \cite[Theorem 3]{tracy2009asymptotics}) instead of Proposition \ref{prop:asepi}, one can recover the main result of \cite{bufetov2022cutoff}, via the same arguments as the proof of Theorem \ref{thm:main}.
Our Proposition \ref{prop:hitrp} plays a similar role as the arguments in Section 3 of \cite{bufetov2022cutoff}, and both use a hitting time bound from \cite{benjamini2005mixing}.
Our Proposition \ref{prop:hec}, on the other hand, plays a similar role as Section 4 of \cite{bufetov2022cutoff}, but its proof (to be given in the next section) is quite different. As discussed above, while both our Proposition \ref{prop:hec} and \cite{bufetov2022cutoff} use Hecke algebra as the main tool, it is used in quite different ways with different points of view, enabling us to analyze this more delicate biased care shuffling.
\end{rem}
The remaining three sections are devoted to proving the three propositions.

\section{ASEP on the line and an interval: Hecke algebra} \label{sec:halg}
In this section we prove Proposition \ref{prop:hec}.
We can write $\cA_{N,k,t}^{a,b}=\cA_{N,k,t}^{a,b,+}\cup\cA_{N,k,t}^{a,b,-}$, where
\begin{itemize}
    \item[$\cA_{N,k,t}^{a,b,-}$:] $h\{\zeta_t^k\}(N-k)>N-k+b$ and
$h\{\xi_t^k\}(N-k-a)<N-k-a$,
    \item[$\cA_{N,k,t}^{a,b,+}$:] $h\{\zeta_t^k\}(N-k)>N-k+b$ and
$h\{\xi_t^k\}(N-k+a)<N-k-a$.
\end{itemize}
By symmetry, below we just bound the probability $\PP[\cA_{N,k,t}^{a,b,+}]$.

\subsection{Hecke algebra and basic properties}

We start by formally introducing Hecke algebra, the main tool of this section. Several notations below are from \cite{bufetov2022cutoff}.

For any integers $m< n$, we can think of $S_{m,n}$ as the permutation group of $\llbracket m,n\rrbracket$, such that for any $w, v\in S_{m,n}$ we let $wv=v\circ w$.
The Hecke algebra $\cH_{m,n}$ is the algebra with basis $\{T_w\}_{w\in S_{m,n}}$, and the rules
\begin{equation}  \label{eq:heckedef}
T_sT_w=
\begin{cases}
T_{sw}, & \text{if } \kappa(sw)=\kappa(w)-1, \\
(1-q)T_w+qT_{sw}, & \text{if } \kappa(sw)=\kappa(w)+1,
\end{cases}
\end{equation}
for any $w\in S_{m,n}$ and $s$ being any nearest neighbor transposition in $S_{m,n}$ (i.e. there is some $i\in\llbracket m,n-1\rrbracket$ such that $s(i)=i+1$, $s(i+1)=i$, and $s(j)=j$ for any $j\in\llbracket m,n\rrbracket\setminus\{i,i+1\}$).
Here we recall (from Section \ref{ssec:mallows}) that $\kappa(w)=\sum_{m\le i<j\le n}\don[w(i)<w(j)]$ is the `energy' for any $w\in S_{m,n}$.

We let $\cH_{m,n}^{\pr}\subset\cH_{m,n}$ be the \emph{probability sub-algebra}, containing all $\sum_{w\in S_{m,n}} p_wT_w$ with $\sum_{w\in S_{m,n}} p_w=1$ and each $p_w\ge 0$.
Then any element in $\cH_{m,n}^{\pr}$ corresponds to a probability measure of $S_{m,n}$.

Note that for any $m'\le m\le n \le n'$, $S_{m,n}$ is naturally embedded into $S_{m',n'}$.
This gives a natural embedding of $\cH_{m,n}$ (resp. $\cH_{m,n}^{\pr}$) into $\cH_{m',n'}$ (resp. $\cH_{m',n'}^{\pr}$).\\

\noindent\textbf{Biased card shuffling.} From \eqref{eq:heckedef}, it is not difficult to see that the evolution of the biased card shuffling can be written as multiplying a uniformly chosen random transposition, at rate $n-m$.
Namely, we denote $\hT_{m,n}=(n-m)^{-1}\sum_{i=m}^{n-1}T_{s_i}$, with $s_i$ being the transposition between $i$ and $i+1$. Let
\[
W_{m,n}(t)=e^{-(n-m)t}\sum_{i=0}^\infty \frac{((n-m)t)^i}{i!} \hT_{m,n}^i,
\]
where $\hT_{m,n}^i$ is the product of $i$ copies of $\hT_{m,n}$ when $i\ge 1$, and $\hT_{m,n}^0=T_{id}$ for $id\in S_{m,n}$ being the identity element (i.e. the identity map of $\llbracket m,n\rrbracket$).
Then $W_{m,n}(t)$ is in $\cH_{m,n}^{\pr}$, and its corresponding probability measure is just the time $t$ distribution of the ASEP on $\llbracket m,n\rrbracket$, starting from the identity element.\\

\noindent\textbf{Mallows element.} A particularly useful object is the \emph{Mallows element} $M_{m,n}$ of $\cH_{m,n}^{\pr}$, which is defined as
\[
M_{m,n}=\sum_{w\in S_{m,n}} \cM_{m,n}(w) T_w= \sum_{w\in S_{m,n}} q^{\kappa(w)}Z_{m,n} T_w,
\]
where $Z_{m,n}$ is defined in \ref{eq:defz}.
Then the probability measure given by $M_{m,n}$ is just the Mallows measure $\cM_{m,n}$.

Note that since $\cM_{m,n}$ is the stationary measure of the biased card shuffling, we must have that $W_{m,n}(t)\to M_{m,n}$ as $t\to\infty$.\\

\noindent\textbf{Involution.}
Let $\ii:\cH_{m,n}\to\cH_{m,n}$ be the linear map where $\ii(T_w)=T_{w^{-1}}$ for any $w\in S_{m,n}$.
It is straightforward to check that $\ii$ is an involutive anti-homomorphism; namely, for any $T_1, T_2 \in \cH_{m,n}$ we have $\ii(T_1T_2)=\ii(T_2)\ii(T_1)$, and $\ii(\ii(T_1))=T_1$ (see e.g. \cite[Proposition 5.1]{bufetov2022cutoff}).

We also have that $\ii(W_{m,n}(t))=W_{m,n}(t)$, since $s=s^{-1}$ for any nearest neighbor transposition $s$ in $S_{m,n}$.
We note that this equality can be interpreted as the ASEP `color-to-position symmetry', as proved in \cite{amir2011tasep, angel2009oriented, borodin2021color}.
By sending $t\to\infty$ we further have $\ii(M_{m,n})=M_{m,n}$.

\subsection{Two additional processes}
We now explain our strategy of bounding $\PP[\cA_{N,k,t}^{a,b,+}]$ (thus proving Proposition \ref{prop:hec}).

In $\cA_{N,k,t}^{a,b,+}$, we consider both $\bzeta^k$, a process on $\Z$, and $\bxi^k$, a process on $[N]$.
To relate them, we shall extend $\bxi^k$, by considering a process on $\Z$, which, roughly speaking, is the same as $\bxi^k$ is $[N]$, while equals $0$ is $\llbracket-\infty,0\rrbracket$ and equals $1$ on $\llbracket N+1,\infty\rrbracket$.
This process is also coupled with $\bzeta^k$ under the basic coupling, and we shall compare it with $\bzeta^k$.

For the comparison, note that the initial configuration of this `extension' of $\bxi^k$ is also (roughly) the same as $\zeta^k_0$ in $[N]$, while different in $\llbracket-\infty,0\rrbracket$ and $\llbracket N+1,\infty\rrbracket$.
We then consider an intermediate process, which is also coupled with $\bxi^k$ and $\bzeta^k$ under the basic coupling, and its initial configuration (roughly) is the same as $\zeta^k_0$ in $\llbracket 1,\infty\rrbracket$, and equals $0$ in $\llbracket-\infty,0\rrbracket$.
This intermediate process together with $\bzeta^k$ can be encoded as one ASEP, by placing a second-class particle at each location where they differ.
Then comparing these two processes is reduced to tracking the locations of the second-class particles, and that can be achieved using the above involution of Hecke algebra.
The comparison between the intermediate process and the `extension' of $\bxi^k$ is analyzed similarly. Putting these together we get estimates on comparing $\bzeta^k$ and $\bxi^k$. See Figure \ref{fig:cp} for an illustration of these processes, whose formal definitions are given below.
\begin{figure}[hbt!]
    \centering
\begin{tikzpicture}[line cap=round,line join=round,>=triangle 45,x=3.7cm,y=3.7cm]
\clip(-3.86,-0.7) rectangle (0.75,2.45);

\draw [line width=3.4pt, color=blue, opacity=0.3] plot coordinates {(-2.95,2.35) (-1.,0.4) (8,9.4)};
\draw [line width=2.2pt, color=yellow] plot coordinates {(-2.95,0.05) (-2,1) (-1.9,0.9) (-1.8,1) (-1.6,0.8) (-1.5,0.9) (-1.,0.4) (-0.4,1.) (-0.3,0.9) (-0.2,1) (1,-0.2)};
\draw [line width=1.1pt, color=brown] plot coordinates {(-2.95,0.05) (-2,1) (-1.9,0.9) (-1.8,1) (-1.6,0.8) (-1.5,0.9) (-1.,0.4) (8,9.4) };
\draw [line width=0.3pt, color=uuuuuu] plot coordinates {(-2.95,0.05) (-1.8,1.2) (-1.,0.4) (-0.3,1.1) (1,-0.2)};

\foreach \i in {1,...,39}
{
\draw [line width=0.5pt, opacity=0.2] [dashed] (\i/10-3,-0.6) -- (\i/10-3,5.6);
}
\foreach \i in {-11,...,14}
{
\draw [line width=0.5pt, opacity=0.2] [dashed] (-2.95,\i/10+1.2) -- (1,\i/10+1.2);
\begin{tiny}
\draw (-2.95,\i/10+1.2) node[anchor=east]{$\i$};
\end{tiny}
}

\draw [line width=1.8pt, opacity=0.5] (-2.95,-0.1) -- (0.95,-0.1);
\draw [fill=uuuuuu] (-2.9,-0.1) circle (3.0pt);
\draw [fill=uuuuuu] (-2.8,-0.1) circle (3.0pt);
\draw [fill=uuuuuu] (-2.7,-0.1) circle (3.0pt);
\draw [fill=uuuuuu] (-2.6,-0.1) circle (3.0pt);
\draw [fill=uuuuuu] (-2.5,-0.1) circle (3.0pt);
\draw [fill=uuuuuu] (-2.4,-0.1) circle (3.0pt);
\draw [fill=uuuuuu] (-2.3,-0.1) circle (3.0pt);
\draw [fill=uuuuuu] (-2.2,-0.1) circle (3.0pt);
\draw [fill=uuuuuu] (-2.1,-0.1) circle (3.0pt);
\draw [fill=uuuuuu] (-2.,-0.1) circle (3.0pt);
\draw [fill=uuuuuu] (-1.9,-0.1) circle (3.0pt);
\draw [fill=uuuuuu] (-1.8,-0.1) circle (3.0pt);
\draw [fill=uuuuuu] (-1.7,-0.1) circle (3.0pt);
\draw [fill=uuuuuu] (-1.6,-0.1) circle (3.0pt);
\draw [fill=uuuuuu] (-1.5,-0.1) circle (3.0pt);
\draw [fill=uuuuuu] (-1.4,-0.1) circle (3.0pt);
\draw [fill=uuuuuu] (-1.3,-0.1) circle (3.0pt);
\draw [fill=uuuuuu] (-1.2,-0.1) circle (3.0pt);
\draw [fill=uuuuuu] (-1.1,-0.1) circle (3.0pt);
\draw [fill=uuuuuu] (-1.,-0.1) circle (3.0pt);
\draw [fill=white] (-0.9,-0.1) circle (3.0pt);
\draw [fill=white] (-0.8,-0.1) circle (3.0pt);
\draw [fill=white] (-0.7,-0.1) circle (3.0pt);
\draw [fill=white] (-0.6,-0.1) circle (3.0pt);
\draw [fill=white] (-0.5,-0.1) circle (3.0pt);
\draw [fill=white] (-0.4,-0.1) circle (3.0pt);
\draw [fill=white] (-0.3,-0.1) circle (3.0pt);
\draw [fill=white] (-0.2,-0.1) circle (3.0pt);
\draw [fill=white] (-0.1,-0.1) circle (3.0pt);
\draw [fill=white] (0.,-0.1) circle (3.0pt);
\draw [fill=white] (0.1,-0.1) circle (3.0pt);
\draw [fill=white] (0.2,-0.1) circle (3.0pt);
\draw [fill=white] (0.3,-0.1) circle (3.0pt);
\draw [fill=white] (0.4,-0.1) circle (3.0pt);
\draw [fill=white] (0.5,-0.1) circle (3.0pt);
\draw [fill=white] (0.6,-0.1) circle (3.0pt);
\draw [fill=white] (0.7,-0.1) circle (3.0pt);
\draw [fill=white] (0.8,-0.1) circle (3.0pt);
\draw [fill=white] (0.9,-0.1) circle (3.0pt);

\draw [line width=1.8pt, opacity=0.5] (-2.95,-0.25) -- (0.95,-0.25);
\draw [fill=white] (-2.9,-0.25) circle (3.0pt);
\draw [fill=white] (-2.8,-0.25) circle (3.0pt);
\draw [fill=white] (-2.7,-0.25) circle (3.0pt);
\draw [fill=white] (-2.6,-0.25) circle (3.0pt);
\draw [fill=white] (-2.5,-0.25) circle (3.0pt);
\draw [fill=white] (-2.4,-0.25) circle (3.0pt);
\draw [fill=white] (-2.3,-0.25) circle (3.0pt);
\draw [fill=white] (-2.2,-0.25) circle (3.0pt);
\draw [fill=white] (-2.1,-0.25) circle (3.0pt);
\draw [fill=white] (-2.,-0.25) circle (3.0pt);
\draw [fill=uuuuuu] (-1.9,-0.25) circle (3.0pt);
\draw [fill=white] (-1.8,-0.25) circle (3.0pt);
\draw [fill=uuuuuu] (-1.7,-0.25) circle (3.0pt);
\draw [fill=uuuuuu] (-1.6,-0.25) circle (3.0pt);
\draw [fill=white] (-1.5,-0.25) circle (3.0pt);
\draw [fill=uuuuuu] (-1.4,-0.25) circle (3.0pt);
\draw [fill=uuuuuu] (-1.3,-0.25) circle (3.0pt);
\draw [fill=uuuuuu] (-1.2,-0.25) circle (3.0pt);
\draw [fill=uuuuuu] (-1.1,-0.25) circle (3.0pt);
\draw [fill=uuuuuu] (-1.,-0.25) circle (3.0pt);
\draw [fill=white] (-0.9,-0.25) circle (3.0pt);
\draw [fill=white] (-0.8,-0.25) circle (3.0pt);
\draw [fill=white] (-0.7,-0.25) circle (3.0pt);
\draw [fill=white] (-0.6,-0.25) circle (3.0pt);
\draw [fill=white] (-0.5,-0.25) circle (3.0pt);
\draw [fill=white] (-0.4,-0.25) circle (3.0pt);
\draw [fill=white] (-0.3,-0.25) circle (3.0pt);
\draw [fill=white] (-0.2,-0.25) circle (3.0pt);
\draw [fill=white] (-0.1,-0.25) circle (3.0pt);
\draw [fill=white] (0.,-0.25) circle (3.0pt);
\draw [fill=white] (0.1,-0.25) circle (3.0pt);
\draw [fill=white] (0.2,-0.25) circle (3.0pt);
\draw [fill=white] (0.3,-0.25) circle (3.0pt);
\draw [fill=white] (0.4,-0.25) circle (3.0pt);
\draw [fill=white] (0.5,-0.25) circle (3.0pt);
\draw [fill=white] (0.6,-0.25) circle (3.0pt);
\draw [fill=white] (0.7,-0.25) circle (3.0pt);
\draw [fill=white] (0.8,-0.25) circle (3.0pt);
\draw [fill=white] (0.9,-0.25) circle (3.0pt);

\draw [line width=1.8pt, opacity=0.5] (-2.95,-0.4) -- (0.95,-0.4);
\draw [fill=white] (-2.9,-0.4) circle (3.0pt);
\draw [fill=white] (-2.8,-0.4) circle (3.0pt);
\draw [fill=white] (-2.7,-0.4) circle (3.0pt);
\draw [fill=white] (-2.6,-0.4) circle (3.0pt);
\draw [fill=white] (-2.5,-0.4) circle (3.0pt);
\draw [fill=white] (-2.4,-0.4) circle (3.0pt);
\draw [fill=white] (-2.3,-0.4) circle (3.0pt);
\draw [fill=white] (-2.2,-0.4) circle (3.0pt);
\draw [fill=white] (-2.1,-0.4) circle (3.0pt);
\draw [fill=white] (-2.,-0.4) circle (3.0pt);
\draw [fill=uuuuuu] (-1.9,-0.4) circle (3.0pt);
\draw [fill=white] (-1.8,-0.4) circle (3.0pt);
\draw [fill=uuuuuu] (-1.7,-0.4) circle (3.0pt);
\draw [fill=uuuuuu] (-1.6,-0.4) circle (3.0pt);
\draw [fill=white] (-1.5,-0.4) circle (3.0pt);
\draw [fill=uuuuuu] (-1.4,-0.4) circle (3.0pt);
\draw [fill=uuuuuu] (-1.3,-0.4) circle (3.0pt);
\draw [fill=uuuuuu] (-1.2,-0.4) circle (3.0pt);
\draw [fill=uuuuuu] (-1.1,-0.4) circle (3.0pt);
\draw [fill=uuuuuu] (-1.,-0.4) circle (3.0pt);
\draw [fill=white] (-0.9,-0.4) circle (3.0pt);
\draw [fill=white] (-0.8,-0.4) circle (3.0pt);
\draw [fill=white] (-0.7,-0.4) circle (3.0pt);
\draw [fill=white] (-0.6,-0.4) circle (3.0pt);
\draw [fill=white] (-0.5,-0.4) circle (3.0pt);
\draw [fill=white] (-0.4,-0.4) circle (3.0pt);
\draw [fill=uuuuuu] (-0.3,-0.4) circle (3.0pt);
\draw [fill=white] (-0.2,-0.4) circle (3.0pt);
\draw [fill=uuuuuu] (-0.1,-0.4) circle (3.0pt);
\draw [fill=uuuuuu] (0.,-0.4) circle (3.0pt);
\draw [fill=uuuuuu] (0.1,-0.4) circle (3.0pt);
\draw [fill=uuuuuu] (0.2,-0.4) circle (3.0pt);
\draw [fill=uuuuuu] (0.3,-0.4) circle (3.0pt);
\draw [fill=uuuuuu] (0.4,-0.4) circle (3.0pt);
\draw [fill=uuuuuu] (0.5,-0.4) circle (3.0pt);
\draw [fill=uuuuuu] (0.6,-0.4) circle (3.0pt);
\draw [fill=uuuuuu] (0.7,-0.4) circle (3.0pt);
\draw [fill=uuuuuu] (0.8,-0.4) circle (3.0pt);
\draw [fill=uuuuuu] (0.9,-0.4) circle (3.0pt);

\draw [line width=1.8pt, opacity=0.5] (-2.95,-0.55) -- (0.95,-0.55);
\draw [fill=uuuuuu] (-1.7,-0.55) circle (3.0pt);
\draw [fill=uuuuuu] (-1.6,-0.55) circle (3.0pt);
\draw [fill=uuuuuu] (-1.5,-0.55) circle (3.0pt);
\draw [fill=uuuuuu] (-1.4,-0.55) circle (3.0pt);
\draw [fill=uuuuuu] (-1.3,-0.55) circle (3.0pt);
\draw [fill=uuuuuu] (-1.2,-0.55) circle (3.0pt);
\draw [fill=uuuuuu] (-1.1,-0.55) circle (3.0pt);
\draw [fill=uuuuuu] (-1.,-0.55) circle (3.0pt);
\draw [fill=white] (-0.9,-0.55) circle (3.0pt);
\draw [fill=white] (-0.8,-0.55) circle (3.0pt);
\draw [fill=white] (-0.7,-0.55) circle (3.0pt);
\draw [fill=white] (-0.6,-0.55) circle (3.0pt);
\draw [fill=white] (-0.5,-0.55) circle (3.0pt);
\draw [fill=white] (-0.4,-0.55) circle (3.0pt);
\draw [fill=white] (-0.3,-0.55) circle (3.0pt);

\begin{scriptsize}
\draw (-3,-0.1) node[anchor=east]{$\etao_0^k$};
\draw (-3,-0.25) node[anchor=east]{$\etao_0^{-X+k-1}=\etad_0^{-X+k-1}$};
\draw (-3,-0.4) node[anchor=east]{$\etad^{-1+k-N}_0$};
\draw (-3,-0.55) node[anchor=east]{$\xi_0^k$};
\end{scriptsize}

\foreach \i in {-11,...,35}
{
\begin{tiny}
\draw (-1.8+\i/10,-0.7) node[anchor=south]{$\i$};
\end{tiny}
}

\end{tikzpicture}
\caption{
An illustration of various processes and their height functions: the black, yellow, brown, blue functions are $h\{\xi_0^k\}$, $h\{\etad^{-1+k-N}_0\}$, $h\{\etao_0^{-X+k-1}\}=h\{\etad_0^{-X+k-1}\}$, and $h'\{\etao_0^k\}$, respectively.
The process $\btad^{-1+k-N}$ is the `extension' of $\bxi^k$, and $\btad^{-X+k-1}=\btao^{-X+k-1}$ is the `intermediate process'.
}  
\label{fig:cp}
\end{figure}
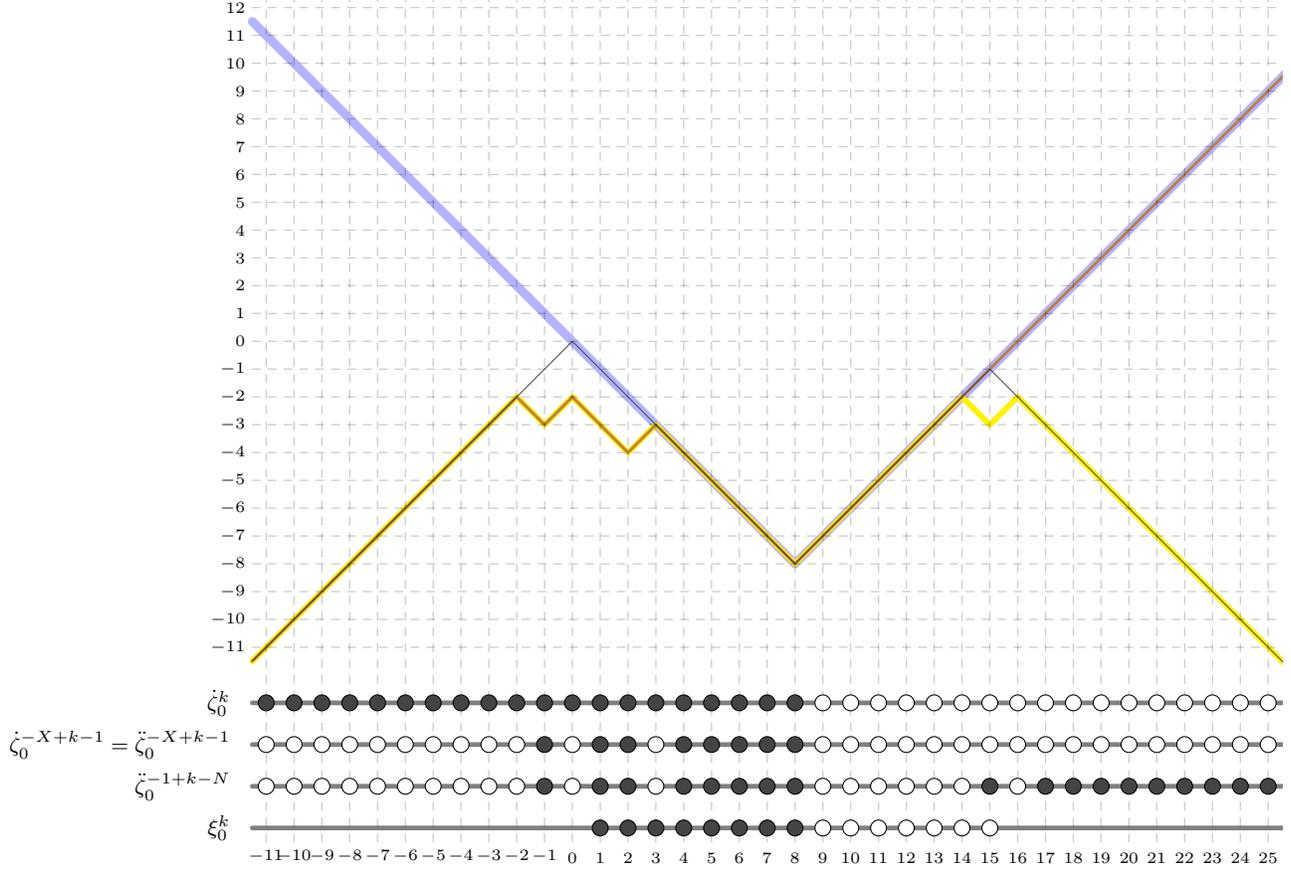

We start by defining the `extension' of $\bxi^k$ and the intermediate process, using the Hecke algebra Mallows elements.
There are two key aspects to note. First, to save notations, we do not formally use the notion of second-class particles, and just write our proof using the multi-species ASEP setup.
Second, to work with the Hecke algebra, we shall consider a large finite interval instead of $\Z$.

Take an integer $X$ that is large enough depending on $N, k, a, b$. 
We let $\btao=(\etao_t)_{t\ge 0}$ and $\btad=(\etad_t)_{t\ge 0}$ be multi-species TASEPs on $\llbracket -X, X\rrbracket$, such that the distribution of $\etao_0$ is given by $M_{-X,k}$, and the distribution of $\etad_0$ is given by $M_{-X,k}M_{-X+k,X}$ (both as elements of $\cH_{-X,X}^{\pr}$).
Thus $\etao_0$ and $\etad_0$ can be described as follows:
\begin{itemize}
    \item[$\etao_0$:] In $\llbracket -X,k\rrbracket$, roughly it is decreasing from $k$ to $-X$ (and the distribution is just $\cM_{-X,k}$). In $\llbracket k+1, X\rrbracket$ it is the identity map.
    \item[$\etad_0$:] It can be equivalently defined as follows: first take the random configuration which is the identity map in $\llbracket -X,-X+k-1\rrbracket$ and has distribution $\cM_{-X+k,X}$ in $\llbracket -X+k,X\rrbracket$; then run the ASEP on $\llbracket -X,k\rrbracket$ to time infinity (this procedure is called `bringing into $Q$-equilibrium' in \cite[Section 5.3]{bufetov2022cutoff}).
    Thus, in $\llbracket -X,0\rrbracket$, roughly $\etad_0$ is decreasing from $X$ to $1$; in $\llbracket 1,k\rrbracket$, roughly $\etad_0$ is decreasing from $-X+k-1$ to $-X$; in $\llbracket k+1,X\rrbracket$, roughly $\etad_0$ is decreasing from $0$ to $-X+k$.
\end{itemize}
The process $\btao$ encodes the coupling between $\bzeta^k$ and the intermediate process, and $\btad$ encodes the coupling between the intermediate process and the `extension' of $\bxi^k$.
Namely, if we denote
\[
\etao_t^i(x)=\don[\etao_t(x)\le i],\quad
\etad_t^i(x)=\don[\etad_t(x)\le i],
\]
for each $i\in\llbracket -X, X\rrbracket$, we make the following quick observations.
\begin{enumerate}
    \item $\etao_0^k(x)=\don[x\le k]$ for $x\in\llbracket-X, X\rrbracket$. In particular, the limit of $\etao_0^k$ as $X\to\infty$ is $\zeta_0^k$.
    \item $\etao_0^{-X+k-1}$ and $\etad_0^{-X+k-1}$ are equal in distribution: it is straightforward to check that either of them in $\llbracket -X,k\rrbracket$ has distribution $\cP^k_{-X,k}$, and either of them is $0$ at every location in $\llbracket k+1,X\rrbracket$.
    We couple $\etao_0^{-X+k-1}$ and $\etad_0^{-X+k-1}$ so that they equal almost surely.
    This is the `intermediate process'.
    \item We also consider $\etad^{-1+k-N}_0$. Roughly it equals $0$ in $\llbracket -X,0\rrbracket\cup \llbracket k+1,N\rrbracket$, and equals $1$ in  $\llbracket 1,k\rrbracket\cup \llbracket N+1,\infty\rrbracket$.
    It is the `extension of $\xi^k_0$'.
\end{enumerate}
See Figure \ref{fig:cp} for an illustration of these processes. Given the initial configurations, we couple the evolution of $\btao$, $\btad$ with $\bzeta$, $\bxi$, so that they are all under the same basic coupling (i.e. generated by the same Poisson point processes on $\Z\times\R_{\ge 0}$).

For technical reasons we need to slightly extend the definition of the height function.
For any integers $m<n$ and $\omega:\llbracket m, n\rrbracket\to \{0,1\}$, we define $h'\{\omega\}=h\{\omega'\}$, where $\omega':\Z\to\{0,1\}$ is the function satisfying $\omega'(x)=\omega'(x)$ for $x\in\llbracket m, n\rrbracket$, and $\omega'(x)=1$ for $x\in\Z\setminus\llbracket m, n\rrbracket$. 
Note that compared with $h\{\omega\}$, the only difference is that we let the slope in $\llbracket -\infty, m\rrbracket$ be $-1$ rather than $1$.

We then have the following ordering property (see Figure \ref{fig:cp}).
\begin{lemma}  \label{lem:hei-ord2}
We have $h\{\etad_t^{-1+k-N}\}\le h\{\xi_t^k\}$ and $h\{\etad_t^{-1+k-N}\}\le h\{\etad_t^{-X+k-1}\}=h\{\etao_t^{-X+k-1}\}\le h'\{\etao_t^k\}$ for any $t\ge 0$.
\end{lemma}
\begin{proof}
We first check that 
\begin{equation} \label{eq:hbd01}
h\{\etad_0^{-1+k-N}\}\le h\{\xi_0^k\}.    
\end{equation}
Note that for any $x\in\Z$, we always have that $h\{\etad_0^{-1+k-N}\}(x)\le x\wedge (2N-2k-x)$, since (by definition) we have $h\{\etad_0^{-1+k-N}\}(-X)=-X$, and $\sum_{x\in\llbracket -X,X\rrbracket}h\{\etad_0^{-1+k-N}\}(x) = X-N+k$ thus $h\{\etad_0^{-1+k-N}\}(X)= 2N-2k-X$, and that $h\{\etad_0^{-1+k-N}\}$ is $1$-Lipschitz.
From the definition of $\etad_0$, for each $x\in\llbracket -X,-X+k-1\rrbracket$, we can find $y\in\llbracket -X,k\rrbracket$ with $\etad_0(y)=x$.
This implies that 
\[\sum_{x\in\llbracket -X,k\rrbracket}h\{\etad_0^{-1+k-N}\}(x)\ge k,
\]
so we have $h\{\etad_0^{-1+k-N}\}(k)\le -k$.
Then for each $x\in\llbracket 1,k\rrbracket$ there is $h\{\etad_0^{-1+k-N}\}(x)\le -x$, and for each $x\in \llbracket k+1,N\rrbracket$ there is $h\{\etad_0^{-1+k-N}\}(x)\le x-2k$ (since $h\{\etad_0^{-1+k-N}\}$ is $1$-Lipschitz).
Putting all the bounds together, we get \eqref{eq:hbd01}, 
since $h\{\xi_0^k\}(x)=x\wedge (2N-2k-x)$ for $x\in\llbracket -\infty,0\rrbracket\cup\llbracket N,\infty\rrbracket$, $h\{\xi_0^k\}(x)=-x$ for $x\in\llbracket 0,k\rrbracket$, and $h\{\xi_0^k\}(x)=x-2k$ for $x\in\llbracket k,N\rrbracket$.

We next check that
\begin{equation} \label{eq:hbd02}
h\{\etad_0^{-1+k-N}\}\le h\{\etad_0^{-X+k-1}\}=h\{\etao_0^{-X+k-1}\}\le h'\{\etao_0^k\}.
\end{equation}
The first inequality holds since $\sum_{y\in\llbracket -X,x\rrbracket}\etad_0^{-1+k-N}(y)\ge \sum_{y\in\llbracket -X,x\rrbracket}\etad_0^{-X+k-1}(y)$, for any $x\in\llbracket -X,X\rrbracket$.
The equality is by the second point above. 

For the second inequality, by the second point above we have that 
\[\sum_{x\in\llbracket -X,k\rrbracket}h\{\etao_0^{-X+k-1}\}(x)=k,
\]
so we have $h\{\etao_0^{-X+k-1}\}(k)=-k$.
Then for each $x\in\llbracket 1,k\rrbracket$ there is $h\{\etao_0^{-X+k-1}\}(x)\le -x$, and for each $x\in \llbracket k+1,X\rrbracket$ there is $h\{\etao_0^{-X+k-1}\}(x)\le x-2k$ (since $h\{\etao_0^{-X+k-1}\}$ is $1$-Lipschitz).
We also have that $h\{\etao_0^{-X+k-1}\}(x)\le x\wedge (2X-2k-x)$ for any $x\in\Z$, since $\sum_{x\in\llbracket -X,X\rrbracket}h\{\etao_0^{-X+k-1}\}(x) = k$ (thus $h\{\etao_0^{-X+k-1}\}(X)=X-2k$), and that $h\{\etao_0^{-X+k-1}\}$ is $1$-Lipschitz.
Putting all the bounds together, we get the second inequality in \eqref{eq:hbd01}, 
since $h'\{\etao_0^k\}(x)=-x$ for $x\in\llbracket -\infty,k\rrbracket$, $h'\{\etao_0^k\}(x)=x-2k$ for $x\in\llbracket k+1,X\rrbracket$, and $h'\{\etao_0^k\}(x)=2X-2k-x$ for $x\in\llbracket X,\infty\rrbracket$.

Finally, using arguments similar to those in the proof of Lemma \ref{lem:hei-ord1}, ordering of these height functions is preserved under the basic coupling, so the conclusion follows.
\end{proof}

We now define the events
\begin{itemize}
    \item[$\cEo$:] $h'\{\etao_t^k\}(N-k)>N-k+b$ and $h\{\etao_t^{-X+k-1}\}(N-k+a)<N-k+a$.
    \item[$\cEd$:] $h\{\etad_t^{-X+k-1}\}(N-k+a)= N-k+a$ and $h\{\etad_t^{-1+k-N}\}(N-k+a)<N-k-a$.
\end{itemize}
By Lemma \ref{lem:hei-ord2}, and the fact that the $X\to\infty$ limit of $\etao_t^k$ is $\zeta_t^k$, we have that the $X\to\infty$ limit of $\cEo\cup\cEd$ contains $\cA_{N,k,t}^{a,b,+}$.

Thus it suffices to bound the probabilities of the events $\cEo$ and $\cEd$.
For this, note that using above stated properties of the involution $\ii$, we have
\begin{equation}  \label{eq:oseq}
W_{-X,X}(t)M_{-X,k}=\ii(W_{-X,X}(t))\ii(M_{-X,k})=\ii(M_{-X,k}W_{-X,X}(t)),    
\end{equation}
and
\begin{equation}  \label{eq:dseq}
W_{-X,X}(t)M_{-X,k}M_{-X+k,X}=\ii(W_{-X,X}(t))\ii(M_{-X,k})\ii(M_{-X+k,X})=\ii(M_{-X+k,X}M_{-X,k}W_{-X,X}(t)).
\end{equation}
These two identities can be interpreted as distributional identities, and are key inputs in bounding $\PP[\cEo]$ and $\PP[\cEd]$.
\begin{lemma}  \label{lem:odot}
We have $\PP[\cEo]\le C q^{b/4-a/2}$ where $C$ is a constant depending only on $q$.
\end{lemma}
\begin{proof}
Assuming $\cEo$, we must have that $h'\{\etao_t^k\}(N-k+a)>N-k+b-a$ (since $h'\{\etao_t^k\}$ is $1$-Lipschitz) and $h\{\etao_t^{-X+k-1}\}(N-k+a)<N-k+a$. In other words, we have
\[
|\{x\in\Z: N-k+a+1\le x \le X, \etao_t(x)\le k\}|>k-a+b/2,
\]
and
\[
|\{x\in\Z: N-k+a+1\le x \le X, \etao_t(x)\le -X+k-1\}|<k.
\]
Now we use the distributional identity implied by \eqref{eq:oseq}. Let $\uetao_t$ be the inverse of $\etao_t$ (as a group element in $S_{-X,X}$).
Then the above two inequalities are equivalent to
\begin{equation}  \label{eq:uetao1}
|\{x\in\Z: \uetao_t(x)\ge N-k+a+1, -X\le x\le k\}|>k-a+b/2,    
\end{equation}
and
\begin{equation}  \label{eq:uetao2}
|\{x\in\Z: \uetao_t(x)\ge N-k+a+1, -X\le x\le -X+k-1\}|<k,
\end{equation}
respectively. Note that these events only concern $x\mapsto \don[\uetao_t(x)\le N-k+a]$, which is a random configuration in $\{0,1\}^{\llbracket -X,X\rrbracket}$. We next analyze its distribution.

By \eqref{eq:oseq}, the law of $\uetao_t$ is given by $M_{-X,k}W_{-X,X}(t)$, meaning that we can first sample $w$ from the law given by $W_{-X,X}(t)$, then sample $\uetao_t$ from the law given by $M_{-X,k}T_w$.
Equivalently, $\uetao_t$ is the same as $w$ on $\llbracket k+1, X\rrbracket$; and on $\llbracket -X, k\rrbracket$, the law of $\uetao_t$ (conditioned on $w$) is the weak limit of running the (multi-species) ASEP from $w$ to time $\infty$.

Let $\phi:\llbracket -X,k\rrbracket\to\{w(x):x\in\llbracket -X,k\rrbracket\}$ be the unique increasing bijection determined by $w$ in $\llbracket -X,k\rrbracket$.
Then $\uetao_t$ in $\llbracket -X, k\rrbracket$ can be written as $\phi\circ v$, for $v$ sampled from $\cM_{-X,k}$.
Now the event in \eqref{eq:uetao1} is equivalent to that
\[
|\{x\in\Z: \phi(x)\ge N-k+a+1, -X\le x\le k\}|>k-a+b/2.
\]
Since $\phi$ is increasing, this is further equivalent to that
\[
\phi(-\lfloor -a+b/2\rfloor) \ge N-k+a+1.
\]
The event in \eqref{eq:uetao2} is equivalent to that
\[
\exists x\in \llbracket -X, -X+k-1\rrbracket,\quad \uetao_t(x)=\phi(v(x))\le N-k+a.
\]
So (since $\phi$ is increasing) these two events together imply that $v(x)<a-b/2$ for some $x\in \llbracket -X, -X+k-1\rrbracket$, which further implies that $\sum_{-X\le i<j\le k} \don[v(i)<a-b/2, v(j)\ge a-b/2]>b/2-a+1$.
Note that $x\mapsto \don[v(x)<a-b/2]$ has distribution $\cP_{-X,k}^{\lceil a-b/2\rceil+X}$, so by Lemma \ref{lem:invbd}, the probability of this event is at most $Cq^{b/4-a/2}$, for some $C>0$ depending only on $q$.
Thus the conclusion follows.
\end{proof}
The bound for $\PP[\cEd]$ is proved using a similar strategy.
\begin{lemma}  \label{lem:ddot}
We have $\PP[\cEd]\le Cq^{a/2}$ where $C$ is a constant depending only on $q$.
\end{lemma}
\begin{proof}
Assuming $\cEd$, we must have
\[
|\{x\in\Z: N-k+a+1\le x \le X, \etad_t(x)\le -X+k-1\}|=k,
\]
and
\[
|\{x\in\Z: N-k+a+1\le x \le X, \etad_t(x)\le -1+k-N\}|<X-N+k-a.
\]
We let $\uetad_t$ be the inverse of $\etad_t$ (as a group element in $S_{-X,X}$).
Then the above two relations are equivalent to
\begin{equation}  \label{eq:uetad1}
|\{x\in\Z: \uetad_t(x)\ge N-k+a+1, -X\le x\le -X+k-1\}|=k,    
\end{equation}
and
\begin{equation}  \label{eq:uetad2}
|\{x\in\Z: \uetad_t(x)\ge N-k+a+1, -X\le x\le -1+k-N\}|<X-N+k-a,
\end{equation}
respectively. By \eqref{eq:dseq}, the law of $\uetad_t$ is given by $M_{-X+k,X}M_{-X,k}W_{-X,X}(t)$, meaning that we can first sample $w$ from the law given by $W_{-X,X}(t)$, and then sample $\uetad_t$ from the law given by $M_{-X+k,X}M_{-X,k}T_w$.
Equivalently, $\uetad_t$ can be obtained from $w$, by first running the (multi-species) ASEP on $\llbracket -X,k\rrbracket$ to time $\infty$ to get $w'$, then running the (multi-species) ASEP on $\llbracket -X+k,X\rrbracket$ (from $w'$) to time $\infty$ to get $\uetad_t$.

From this construction of $\uetad_t$, we have that $\uetad_t$ and $w'$ are the same in $\llbracket -X, -X+k-1\rrbracket$. Then the event in \eqref{eq:uetad1} is equivalent to that 
\[
\uetad_t(x)=w'(x)\ge N-k+a+1, \quad \forall x\in\llbracket -X, -X+k-1\rrbracket,
\]
and this is further equivalent to that
\begin{equation}  \label{eq:dtxbd}
|\{x\in\llbracket -X+k, X\rrbracket: w'(x)\le N-k+a\}|= N-k+a+X+1.
\end{equation}

Let $\phi:\llbracket -X+k,X\rrbracket\to\{w'(x):x\in\llbracket -X+k,X\rrbracket\}$ be the unique increasing bijection determined by $w'$ in $\llbracket -X+k,X\rrbracket$.
Then $\uetad_t$ in $\llbracket -X+k, X\rrbracket$ can be written as $\phi\circ v$, for $v$ sampled from $\cM_{-X+k,X}$.
Under the event in \eqref{eq:dtxbd}, we have that $\phi(N+a)\le N-k+a$ since $\phi$ is increasing.
The event in \eqref{eq:uetad2} is equivalent to that 
\[
\exists x\in\llbracket k-N,X\rrbracket,\quad\uetad_t(x)=\phi(v(x))\ge N-k+a+1.
\]
Then under the events in \eqref{eq:dtxbd} and \eqref{eq:uetad2}, there exists $x\in\llbracket k-N, X\rrbracket$ such that $v(x)\ge N+a+1$.
This further implies that
\[\sum_{-X+k\le i<j\le X} \don[v(i)\le N+a, v(j)\ge N+a+1]
\ge \sum_{i\in \llbracket -X+k, k-N-1\rrbracket}\don[v(i)\le N+a]
\ge a+1.\]
Note that $x\mapsto \don[v(x)\le N+a]$ has distribution $\cP_{-X+k,X}^{N+a+X-k+1}$, so by Lemma \ref{lem:invbd} the probability of this event is at most $Cq^{a/2}$, for some $C>0$ depending only on $q$.
Thus the conclusion follows.
\end{proof}
Putting all these estimates together we can now upper bound $\PP[\cA_{N,k,t}^{a,b}]$.
\begin{proof}[Proof of Proposition \ref{prop:hec}]
From the definition of $\cEo$ and $\cEd$, we have that $\PP[\cA_{N,k,t}^{a,b,+}]\le \limsup_{X\to\infty}\PP[\cEo\cup\cEd]\le \limsup_{X\to\infty}\PP[\cEo]+\PP[\cEd]$.
Then by Lemmas \ref{lem:odot} and \ref{lem:ddot}, we have $\PP[\cA_{N,k,t}^{a,b,+}]\le Cq^{(b/4-a/2)\wedge a/2}$, for $C$ depending only on $q$.
By symmetry the same upper bound holds for $\PP[\cA_{N,k,t}^{a,b,-}]$.
Thus the conclusion follows.
\end{proof}

\section{Hitting time bound}   \label{sec:hit}
In this section we prove Proposition \ref{prop:hitrp}, for which we need the following known estimate on the `ground state hitting time' of the single-species ASEP.

For any $n\in\Z$ and $m\in \N$, we let $\bta^{n,m}=(\eta^{n,m}_t)_{t\ge 0}$ be the (single-species) ASEP, with initial condition given by
$\eta^{n,m}_0(x)=1$ for $x\in\llbracket n-m+1,n\rrbracket$ or $x\in \llbracket n+m+1,\infty\rrbracket$, and $\eta^{n,m}_0(x)=0$ for $x\in\llbracket -\infty,n-m\rrbracket$ or $x\in \llbracket n+1,n+m\rrbracket$.

The following estimate is implied by \cite[Theorem 1.9]{benjamini2005mixing}, which is also used in \cite{bufetov2022cutoff}.
\begin{lemma} \label{lem:mix05}
There is a constant $C_*>0$ depending only on $q$, such that the following is true.
Let $H>0$ be the smallest number with $\eta^{n,m}_H(x)=\don[x\ge n+1]$ for any $x\in\Z$.
Then $\PP[H>C_*m]<1/m$.
\end{lemma}
Below we also assume that $C_*\in\Z$, since otherwise we can replace it by $\lceil C_*\rceil$.

Recall the setting of Proposition \ref{prop:hitrp} and the parameters $N,k,t,a$.
The idea of proving Proposition \ref{prop:hitrp} is to split the time interval $[0,\log(N)^4]$ into segments each with length $C_*a$, and use Lemma \ref{lem:mix05} to conclude that the probability of not reaching the `ground state' in each interval is at most $1/a$.
Then we can multiply these probabilities for all the intervals, since that the ASEP is Markovian.

For each $i\in \Z_{\ge 0}$, we let $\cE_i$ be the event where
\[
h\{\xi_{t+iC_*a}^k\}(N-k-a)=h\{\xi_{t+iC_*a}^k\}(N-k+a)=N-k-a,
\]
and $\cE_i'$ be the event where
\[
h\{\xi_{t+iC_*a+s}^k\}(N-k)<N-k,\quad \forall s\in [0, C_*a],
\]
i.e. $\xi_{t+iC_*a+s}^k$ is not at the `ground state' for any $s\in [0, C_*a]$.
Here $C_*$ is the constant in Lemma \ref{lem:mix05}.
\begin{lemma}  \label{lem:ceibd}
We have $\PP[\cE_i'\mid\xi_{t+iC_*a}^k]\le 1/a$ for any $\xi_{t+iC_*a}^k$ under the event $\cE_i$.
\end{lemma}
\begin{proof}
Note that $h\{\eta^{N-k,a}_0\}(N-k-a)=h\{\eta^{N-k,a}_0\}(N-k+a)=N-k-a$, and $h\{\eta^{N-k,a}_0\}$ has slope $1$ in $\llbracket -\infty, N-k-a\rrbracket \cup \llbracket N-k, N-k+a\rrbracket$, and has slope $-1$ in $\llbracket N-k-a, N-k\rrbracket \cup \llbracket N-k+a, \infty\rrbracket$.
Thus assuming $\cE_i$, we must have $h\{\eta^{N-k,a}_0\} \le h\{\xi_{t+iC_*a}^k\}$ since $h\{\xi_{t+iC_*a}^k\}$ is $1$-Lipschitz.

We couple $(\xi_{t+iC_*a+s}^k)_{s\ge 0}$ and $\bta^{N-k,a}$ under the basic coupling.
Then $h\{\eta^{N-k,a}_s\} \le h\{\xi_{t+iC_*a+s}^k\}$ for any $s\ge 0$, since (by arguments similar to those in the proof of Lemma \ref{lem:hei-ord1}) ordering of these height functions is preserved under the basic coupling.

Now by Lemma \ref{lem:mix05}, with probability $\ge 1-1/a$ we have that $\eta^{N-k,a}_s(x)=\don[x\ge N-k+1]$ for any $x\in\Z$, thus $h\{\eta^{N-k,a}_s\}(N-k)=N-k$, for some $s\in[0,C_*a]$. This implies that we must have  $h\{\xi_{t+iC_*a+s}^k\}(N-k)=N-k$, so $\cE_i'$ does not hold.
Thus the conclusion follows.
\end{proof}

\begin{proof}[Proof of Proposition \ref{prop:hitrp}]
Under the event $\cB_{N,k,t}^{a,\log(N)^4}$, we must have that $\cE_i$ and $\cE_i'$ hold, for each $i\in\llbracket 0, \log(N)^4/(C_*a)-1 \rrbracket$.
By Lemma \ref{lem:ceibd} and that the ASEP is Makovian, we must have that
\[
\PP\Big[\cE_i'\mid \bigcap_{0\le j\le i}\cE_i \cap \bigcap_{0\le j< i}\cE_i'\Big] \le 1/a.
\]
Thus by taking the product of this over $i\in \llbracket 0, \log(N)^4/(C_*a)-1 \rrbracket$, and
using that $2\le a\le \log(N)^2$, we have $\PP[\cB_{N,k,t}^{a,\log(N)^4}]<a^{-\lfloor \log(N)^4/(C_*a)\rfloor}<e^{-\log(2)\lfloor \log(N)^2/C_*\rfloor}$. So the conclusion follows.
\end{proof}

\section{ASEP shift-invariance and convergence to the KPZ fixed point}   \label{sec:asepk}
In this section we deduce Proposition \ref{prop:asepi}.

The first input we need is the shift-invariance property of the multi-specie ASEP, which is degenerated from the shift-invariance property of the colored six-vertex model from \cite{borodin2019shift, galashin2021symmetries}.
Namely, we obtain the following result.
Recall $\bzeta$ (the multi-species ASEP on $\Z$) and its projections from Section \ref{ssec:mset}.
\begin{prop}  \label{prop:shinv}
For any $N\in\N$ and $b\in\R$, $t\ge 0$, the event
\[
h\{\zeta^k_t\}(N-k)>N-k+b,\quad \forall k\in \llbracket 0,N\rrbracket,
\]
has the same probability as the event
\[
h\{\zeta^0_t\}(N-2k)>N+b,\quad \forall k\in \llbracket 0,N\rrbracket.
\]
\end{prop}
This is deduced directly from a special case of \cite[Theorem 1.2]{borodin2019shift} or \cite[Theorem 1.6]{galashin2021symmetries}, using the limit transition from the colored six-vertex model to the multi-species ASEP. Such transition was first observed in \cite{borodin2016stochastic}, and then proved in \cite{aggarwal2017convergence} in the colorless setting; the colored version follows the same proof.
See also \cite[Section 3.1]{zhang2021shift}.
We omit the proof of Proposition \ref{prop:shinv}.

Now we just need to analyze the process $\bzeta^0=(\zeta^0_t)_{\ge 0}$, which is the single-species ASEP with step initial condition.
For this, we use the convergence of the single-species ASEP to the KPZ fixed point, which is a random process $\fh:\R_{\ge 0}\times \R\to\R$ and can be described as a continuous time Markov chain $\fh(t,\cdot)$ with explicit transition probabilities.
The state space of this Markov chain is the upper semi-continuous (UC) space, containing all upper semi-continuous functions $f:\R\to[-\infty,\infty)$ with $f(x)<C(|x|+1)$ for some $C<\infty$ and $f(x)>-\infty$ for some $x$.
We refer the readers to \cite[Definition 3.12]{matetski2016kpz} for the precise definition of the KPZ fixed point.

We need the following convergence of the single-species ASEP to the KPZ fixed point.
\begin{prop}[\protect{\cite[Theorem 2.2(2)]{quastel2022convergence}}]  \label{prop:qs22}
For each $\epsilon>0$, let $\bta^\epsilon=(\eta^\epsilon_t)_{t\ge 0}$ be the (single-species) ASEP with deterministic initial configuration $\eta^\epsilon_0$, such that $\eta^\epsilon_0(x)$ stabilizes as $x\to-\infty$ (so the its height function is well-defined); and let $g_\epsilon\in\R$.
Let $\fh_0:\R\to\R$ be continuous with $|\fh_0(x)|<C(1+|x|^{1/2})$ for some $C<\infty$.
Suppose that
\[
x\mapsto -\epsilon^{1/2}h\{\eta^\epsilon_0\}(2\epsilon^{-1}x) + g_\epsilon
\]
converges to $\fh_0$, uniformly on compact sets. Then for any $t>0$,
\[
x\mapsto -\epsilon^{1/2}h\{\eta^\epsilon_{2\epsilon^{-3/2}t}\}(2\epsilon^{-1}x) + g_\epsilon + \epsilon^{-1}(1-q)t
\]
converges to $\fh((1-q)t,\cdot)$ in distribution, in the uniform on compact sets topology, with $\fh$ being the KPZ fixed point with initial data $\fh(0,\cdot)=\fh_0$.
Here we regard $h\{\eta^\epsilon_0\}$ and $h\{\eta^\epsilon_{2\epsilon^{-3/2}t}\}$ as functions on $\R$, by linearly interpolating between integers.
\end{prop}
Using Propositions \ref{prop:shinv} and \ref{prop:qs22}, we now prove the following statement, which directly implies Proposition \ref{prop:asepi}.
\begin{lemma}  \label{lem:convtogoe}
For any $\tau\in\R$, we have that
\[
-2^{2/3}N^{-1/3}\min\{h\{\zeta^k_{2(1-q)^{-1}(N+\tau N^{1/3})}\}(N-k)+k: k\in\llbracket 0,N\rrbracket\}+2^{2/3}N^{2/3}+2^{2/3}\tau
\]
converges in distribution to the GOE Tracy-Widom distribution, as $N\to\infty$.
\end{lemma}
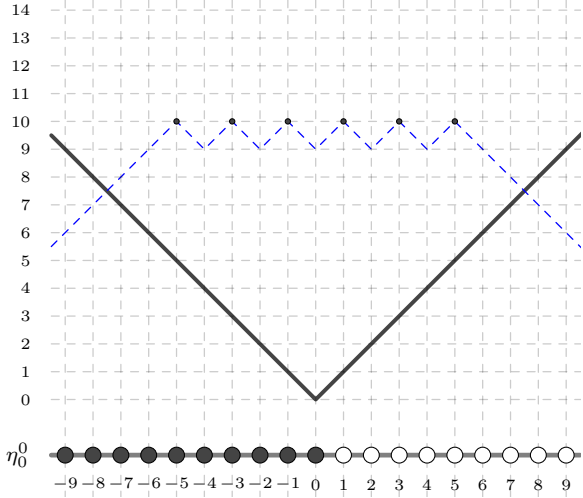
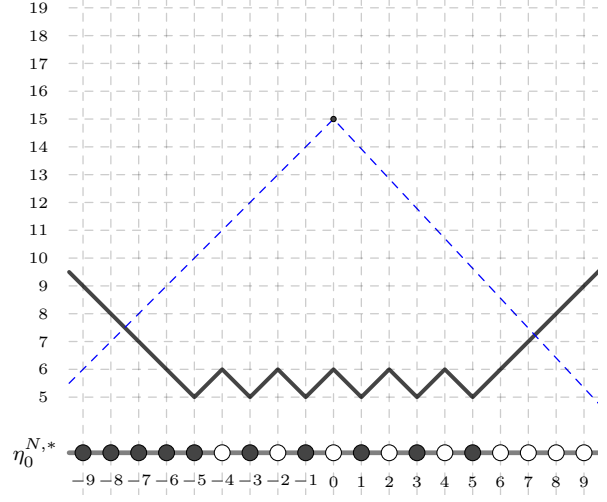
\begin{figure}[hbt!]
    \centering
\hfill
\begin{subfigure}[bt]{0.48\textwidth}
    \centering
\begin{tikzpicture}[line cap=round,line join=round,>=triangle 45,x=3.7cm,y=3.7cm]
\clip(-2.9,-0.25) rectangle (-0.75,1.55);

\draw [line width=1.5pt, color=uuuuuu] plot coordinates {(-2.65,1.05) (-1.7,0.1) (0,1.8)};

\draw [line width=0.5pt, color=blue, style=dashed] plot coordinates {(-2.65,0.65) (-2.2,1.1) (-2.1,1.0) (-2.0,1.1) (-1.9,1.0) (-1.8,1.1) (-1.7,1.0) (-1.6,1.1) (-1.5,1.0) (-1.4,1.1) (-1.3,1.0) (-1.2,1.1) (1,-1.1)};

\foreach \i in {4,...,30}
{
\draw [line width=0.5pt, opacity=0.2] [dashed] (\i/10-3,-0.6) -- (\i/10-3,5.6);
}
\foreach \i in {0,...,14}
{
\draw [line width=0.5pt, opacity=0.2] [dashed] (-2.65,\i/10+0.1) -- (1,\i/10+0.1);
}
\foreach \i in {0,...,19}
{
\begin{tiny}
\draw (-2.7,\i/10+0.1) node[anchor=east]{$\i$};
\end{tiny}
}

\draw [fill=uuuuuu] (-1.8,1.1) circle (1.0pt);
\draw [fill=uuuuuu] (-1.6,1.1) circle (1.0pt);
\draw [fill=uuuuuu] (-1.4,1.1) circle (1.0pt);
\draw [fill=uuuuuu] (-1.2,1.1) circle (1.0pt);
\draw [fill=uuuuuu] (-2.0,1.1) circle (1.0pt);
\draw [fill=uuuuuu] (-2.2,1.1) circle (1.0pt);

\draw [line width=1.8pt, opacity=0.5] (-2.65,-0.1) -- (-0.75,-0.1);
\draw [fill=uuuuuu] (-2.6,-0.1) circle (3.0pt);
\draw [fill=uuuuuu] (-2.5,-0.1) circle (3.0pt);
\draw [fill=uuuuuu] (-2.4,-0.1) circle (3.0pt);
\draw [fill=uuuuuu] (-2.3,-0.1) circle (3.0pt);
\draw [fill=uuuuuu] (-2.2,-0.1) circle (3.0pt);
\draw [fill=uuuuuu] (-2.1,-0.1) circle (3.0pt);
\draw [fill=uuuuuu] (-2.,-0.1) circle (3.0pt);
\draw [fill=uuuuuu] (-1.9,-0.1) circle (3.0pt);
\draw [fill=uuuuuu] (-1.8,-0.1) circle (3.0pt);
\draw [fill=uuuuuu] (-1.7,-0.1) circle (3.0pt);
\draw [fill=white] (-1.6,-0.1) circle (3.0pt);
\draw [fill=white] (-1.5,-0.1) circle (3.0pt);
\draw [fill=white] (-1.4,-0.1) circle (3.0pt);
\draw [fill=white] (-1.3,-0.1) circle (3.0pt);
\draw [fill=white] (-1.2,-0.1) circle (3.0pt);
\draw [fill=white] (-1.1,-0.1) circle (3.0pt);
\draw [fill=white] (-1.,-0.1) circle (3.0pt);
\draw [fill=white] (-0.9,-0.1) circle (3.0pt);
\draw [fill=white] (-0.8,-0.1) circle (3.0pt);
\begin{scriptsize}
\draw (-2.7,-0.1) node[anchor=east]{$\eta_0^0$};
\end{scriptsize}

\foreach \i in {-9,...,9}
{
\begin{tiny}
\draw (-1.7+\i/10,-0.25) node[anchor=south]{$\i$};
\end{tiny}
}

\end{tikzpicture}
\caption{
An illustration of the configuration $\eta_0^0$ and its height function $h\{\eta_0^0\}$.
The above points indicate the lower bounds for $h\{\eta_t^0\}$ given in \eqref{eq:equivr}. 
If these bounds hold, $h\{\eta_t^0\}$ is also lower bounded by the dashed blue function, since it is $1$-Lipschitz.
}  
\end{subfigure}
\hfill
\begin{subfigure}[bt]{0.48\textwidth}
    \centering
\begin{tikzpicture}[line cap=round,line join=round,>=triangle 45,x=3.7cm,y=3.7cm]
\clip(-2.9,-0.25) rectangle (-0.75,1.55);

\draw [line width=1.5pt, color=uuuuuu] plot coordinates {(-2.65,0.55) (-2.2,0.1) (-2.1,0.2) (-2,0.1) (-1.9,0.2) (-1.8,0.1) (-1.7,0.2) (-1.6,0.1) (-1.5,0.2) (-1.4,0.1) (-1.3,0.2) (-1.2,0.1) (3,4.3)};

\draw [line width=0.5pt, color=blue, style=dashed] plot coordinates {(-2.65,0.15) (-1.7,1.1) (1,-1.8)};

\foreach \i in {4,...,30}
{
\draw [line width=0.5pt, opacity=0.2] [dashed] (\i/10-3,-0.6) -- (\i/10-3,5.6);
}
\foreach \i in {5,...,19}
{
\draw [line width=0.5pt, opacity=0.2] [dashed] (-2.65,\i/10-0.4) -- (1,\i/10-0.4);
\begin{tiny}
\draw (-2.7,\i/10-0.4) node[anchor=east]{$\i$};
\end{tiny}
}

\draw [fill=uuuuuu] (-1.7,1.1) circle (1.0pt);

\draw [line width=1.8pt, opacity=0.5] (-2.65,-0.1) -- (-0.75,-0.1);
\draw [fill=uuuuuu] (-2.6,-0.1) circle (3.0pt);
\draw [fill=uuuuuu] (-2.5,-0.1) circle (3.0pt);
\draw [fill=uuuuuu] (-2.4,-0.1) circle (3.0pt);
\draw [fill=uuuuuu] (-2.3,-0.1) circle (3.0pt);
\draw [fill=uuuuuu] (-2.2,-0.1) circle (3.0pt);
\draw [fill=white] (-2.1,-0.1) circle (3.0pt);
\draw [fill=uuuuuu] (-2.,-0.1) circle (3.0pt);
\draw [fill=white] (-1.9,-0.1) circle (3.0pt);
\draw [fill=uuuuuu] (-1.8,-0.1) circle (3.0pt);
\draw [fill=white] (-1.7,-0.1) circle (3.0pt);
\draw [fill=uuuuuu] (-1.6,-0.1) circle (3.0pt);
\draw [fill=white] (-1.5,-0.1) circle (3.0pt);
\draw [fill=uuuuuu] (-1.4,-0.1) circle (3.0pt);
\draw [fill=white] (-1.3,-0.1) circle (3.0pt);
\draw [fill=uuuuuu] (-1.2,-0.1) circle (3.0pt);
\draw [fill=white] (-1.1,-0.1) circle (3.0pt);
\draw [fill=white] (-1.,-0.1) circle (3.0pt);
\draw [fill=white] (-0.9,-0.1) circle (3.0pt);
\draw [fill=white] (-0.8,-0.1) circle (3.0pt);
\begin{scriptsize}
\draw (-2.66,-0.1) node[anchor=east]{$\eta_0^{N,*}$};
\end{scriptsize}

\foreach \i in {-9,...,9}
{
\begin{tiny}
\draw (-1.7+\i/10,-0.25) node[anchor=south]{$\i$};
\end{tiny}
}

\end{tikzpicture}
\caption{
An illustration of the configuration $\eta_0^{N,*}$ and its height function $h\{\eta_0^{N,*}\}$.
The above point indicates the lower bound for $h\{\eta_t^{N,*}\}$ given in \eqref{eq:equivr}. 
If these bounds hold, $h\{\eta_t^{N,*}\}$ is also lower bounded by the dashed blue function, since it is $1$-Lipschitz.
}  
\end{subfigure}
\caption{An illustration of the skew-time reversibility, as given by e.g. (1.4) of \cite{quastel2022convergence}.}
\label{fig:trev}
\end{figure}
\begin{proof}
Let $\bta^{N,*}=(\eta_t^{N,*})_{t\ge 0}$ be the single-species ASEP with the following initial configuration:
$\eta_0^{N,*}(x)=1$ for $x\in \llbracket-\infty, -N\rrbracket$; and $\eta_0^{N,*}(x)=0$ for $x\in \llbracket N+1,\infty\rrbracket$; and $\eta_0^{N,*}(-N-1+2x)=0$, $\eta_0^{N,*}(-N+2x)=1$, for $x\in\llbracket 1,N\rrbracket$.
Thus we have
\[
h\{\eta_0^{N,*}\}(x)=
\begin{cases}
-x, &\text{ when } x\in \llbracket-\infty, -N\rrbracket,\\
N, &\text{ when } x\in 2\llbracket 0, N\rrbracket-N,\\
N+1, &\text{ when } x\in 2\llbracket 0, N-1\rrbracket-N+1,\\
x, &\text{ when } x\in \llbracket N,\infty\rrbracket.\\
\end{cases}
\]
By skew-time reversibility of ASEP (see e.g. (1.4) of \cite{quastel2022convergence}), for any $t\ge 0$ and $b\in\R$, we have
\begin{equation} \label{eq:equivr}
\PP[h\{\zeta^0_t\}(N-2k)>N+b,\; \forall k\in \llbracket 0,N\rrbracket]=\PP[h\{\eta_t^{N,*}\}(0)>2N+b].
\end{equation}
See Figure \ref{fig:trev}. Then by Proposition \ref{prop:shinv} we have
\[
\PP[h\{\zeta^k_t\}(N-k)>N-k+b,\; \forall k\in \llbracket 0,N\rrbracket]=\PP[h\{\eta_t^{N,*}\}(0)>2N+b].    
\]
Namely, we have
\begin{equation}  \label{eq:aseped}
\min\{h\{\zeta^k_t\}(N-k)+k: k\in\llbracket 0,N\rrbracket\}-N \doeq h\{\eta_t^{N,*}\}(0)-2N,
\end{equation}
where $\doeq$ denotes equal in distribution.

Note that for any $\tau\in \R$, we have that
\[
x\mapsto - (N+\tau N^{1/3})^{-1/3} h\{\eta^{N,*}_0\}(2(N+\tau N^{1/3})^{2/3}x) + (N+\tau N^{1/3})^{-1/3}N
\]
converges (uniformly in compact sets) to $0$, as $N\to\infty$.
Then by Proposition \ref{prop:qs22}, we have that
\[
- (N+\tau N^{1/3})^{-1/3} h\{\eta^{N,*}_{2(1-q)^{-1}(N+\tau N^{1/3})}\}(0) +(N+\tau N^{1/3})^{-1/3}N + (N+\tau N^{1/3})^{2/3}
\]
converges in distribution to $\fh(1,0)$ as $N\to\infty$, for $\fh$ being the KPZ fixed point with initial data $\fh(0,\cdot)=0$.
We note that with such initial data, $x\mapsto 2^{-1/3}\fh(1,2^{2/3}x)$ is the so-called Airy$_1$ process, and $\fh(1,0)$ has the same distribution as $\sup_{x\in\R}\cA(x)-x^2$, for $\cA$ being the stationary Airy$_2$ process (see e.g. (4.15) and Example 4.20 in \cite{matetski2016kpz}).
Thus $2^{2/3}\fh(1,0)$ has GOE Tracy-Widom distribution (see e.g. (1.25) of \cite{quastel2014airy}).
By multiplying $2^{2/3}N^{-1/3}(N+\tau N^{1/3})^{1/3}$ we have that
\[
-2^{2/3}N^{-1/3}h\{\eta^{N,*}_{2(1-q)^{-1}(N+\tau N^{1/3})}\}(0)+2^{5/3}N^{2/3}+2^{2/3}\tau
\]
converges in distribution to the GOE Tracy-Widom distribution, as $N\to\infty$.
Using \eqref{eq:aseped} the conclusion follows.
\end{proof}

\subsection*{Acknowledgement}
The author would like to thank Jimmy He and Amol Aggrawal for several valuable conversations.
The research of the author is supported by the Miller Institute for Basic Research in Science, at University of California, Berkeley.
Part of this work was completed when the author was a PhD student at Princeton University, Department of Mathematics.

\bibliographystyle{plain}
\bibliography{bibliography}

\end{document}